\documentclass[final]{siamltex}
\usepackage{enumerate}
\usepackage{color}
\usepackage{amsfonts}
\usepackage{mathrsfs}
\usepackage{mathdots}
\usepackage{verbatim}
\usepackage{graphicx}
\usepackage{subfigure}
\usepackage{amsmath}
\usepackage{multirow}
\usepackage{float}
\usepackage{amssymb}%

\newcounter{algorithm}
\newenvironment{algorithm}{\refstepcounter{algorithm}\vspace{1ex}
{\sc Algorithm \thealgorithm.}\hspace{0.3em}\parindent=0pt}{\vspace{1ex}}
\newcounter{remark}
\newenvironment{remark}{\refstepcounter{remark}\vspace{1ex}
{\sc Remark \theremark.}\hspace{0.3em}\parindent=0pt}{\vspace{1ex}}

\begin{document}

\title{An Improved {dqds} algorithm}
\author{Shengguo Li\thanks{College of Science, and the State key laboratory for high performance
computation, National University of Defense
Technology, Changsha 410073, China (nudtlsg\allowbreak @gmail.com).
The research of Li was supported by CSC (2010611043)
%and in part by the Graduate School of NUDT, Fund of Innovation (B100201),
%the innovation for postgraduate of Hunan province (CX2010B006)
and in part by
National Natural Science Foundation of China (No. 60921062 and 61201328).}
\and Ming Gu\thanks{Department of Mathematics, University of California, Berkeley,
  CA 47920, US (mgu@math.\allowbreak berkeley.edu)
The research of Gu was supported in part by the Director,
Office of Science, Office of Advanced Scientific Computing
Research of the U.S. Department of Energy under Contract
No. DE-AC02-05CH11231, and by NSF Awards CCF-0830764 and and CCF-1319312.}
\and Beresford N. Parlett\thanks{Department of Mathematics, University of California, Berkeley,
CA 47920, US (parlett@math.\allowbreak berkeley.edu).}}
\maketitle

\begin{abstract} In this paper we present an improved dqds algorithm
for computing all the singular values of a bidiagonal matrix to high
relative accuracy. There are two key contributions: a novel deflation
strategy that improves the convergence for badly scaled matrices, and
some modifications to certain shift strategies that accelerate the
convergence for most bidiagonal matrices.  These techniques together
ensure linear worst case complexity of the improved algorithm
(denoted by V5).  Our extensive numerical experiments indicate that V5
is typically 1.2x--4x faster than DLASQ (the LAPACK-3.4.0
implementation of dqds) without any degradation in accuracy. On
matrices for which DLASQ shows very slow convergence, V5 can be
3x--10x faster.  At the end of this paper, a hybrid algorithm (HDLASQ)
is developed by combining our improvements with the aggressive early
deflation strategy (AggDef2 in [SIAM J. Matrix Anal. Appl., 33(2012),
22-51]).  Numerical results show that HDLASQ is the fastest among
these different versions.
\end{abstract}

%\thanks{This work was supported by NSF and NFS}

%The thanks line in the title should be filled in if there is
%any support acknowledgement for the overall work to be included
%This \thanks is also used for the received by date info, but
%authors are not expected to provide this.



\pagestyle{myheadings} \thispagestyle{plain} \markboth{Modified implementation}{dqds algorithm}

\section{Introduction}
The dqds (differential quotient difference with shifts) algorithm
of~\cite{Fernando94} computes the singular values of bidiagonal
matrices to high relative accuracy.  A detailed account of its
efficient implementation, which is essentially the LAPACK code, can be
found in~\cite{Parlett-Imp}. One critical technique responsible for
its success is the rather complex shift strategy. An aggressive early
deflation (AED) strategy is suggested in~\cite{dqds-Agg} to further
enhance performance for large matrices.

Our work is motivated by the recent discovery of significant slowdown
of the LAPACK routine DLASQ\footnote{Version 3.4.0 or earlier. The lastest version of DLASQ has already adapted our
strategies and therefore runs in linear time in the worst case.}  on
certain matrices which have diagonal and off-diagonal entries that are
of varying magnitude and are disordered (far from monotone
decreasing). The improvements of our algorithm are
\begin{enumerate}
\item the ability to deflate converged singular values that are far
from the bottom of the matrix;
\item an improved bound on the current smallest singular value;
\item some modifications to certain shift strategies;
\item a guarantee that the convergence is never slower than that for
the bisection algorithm.
\end{enumerate}

Our improved implementation (denoted by V5) is 1.2x-4x faster than
DLASQ in general and up to 10x faster for matrices for which the dqds
algorithm converges slowly. It should be mentioned that the dqds
algorithm has found other uses than computation of the singular values
and the ideas in this paper may be of value in those situations as
well~\cite{Carla-thesis}.

V5 is never slower than the bisection algorithm in convergence, which
establishes its linear complexity in the worst case and ensures a
more robust implementation of the dqds algorithm. There is a similar
result for the QR algorithm with Wilkinson shift for the symmetric
tridiagonal eigenproblems~\cite{Parlett-book}.

V5 and AED are complementary: AED may help for matrices which are easy
for dqds, whereas V5 improves the convergence for the disordered matrices.
By combining V5 with AED, a more efficient algorithm is obtained in
section~\ref{sec:Num-dqds}.  This hybrid algorithm would be denoted by
HDLASQ and is shown to be faster than DLASQ, V5 and AggDef2 (proposed
in~\cite{dqds-Agg}) for both the easy and difficult matrices.

This paper is organized as follows. Sections~\ref{sec:essential}
and~\ref{sec:lapack} introduce the dqds algorithm and its LAPACK
implementation DLASQ, respectively. Section~\ref{sec:improvement}
discusses the five important improvements over DLASQ.
Section~\ref{sec:finiteCov} proves the linear worst case complexity of
V5.
Section~\ref{sec:Num-dqds} compares V5 with DLASQ and AggDef2, and introduces
the hybrid algorithm HDLASQ.
Numerous experiments on synthetic and practical matrices are also
reported in section~\ref{sec:Num-dqds}.
Section~\ref{sec:conclusion} presents the conclusions.

\section{Essential information on dqds}
\label{sec:essential}

In this section we gather essential properties of the dqds algorithm.
The reader is expected to have some acquaintance with
this algorithm.  For an introduction to the topic we
recommend~\cite{Parlett-Acta} and for historical
details~\cite{Fernando94}. By convention our bidiagonal
matrices are all upper bidiagonal. Our notation is as follows
\begin{equation}
\label{eq:B}
B= \begin{bmatrix}  a_1 & b_1 &&&& \\ & a_2 & b_2 &&& \\ & & \cdotp & \cdotp && \\
    & & & \cdotp & \cdotp & \\ & & & & a_{n-1} &b_{n-1} \\ &&&&& a_n\end{bmatrix},
\end{equation}
and all the elements of $B$ are positive.

In addition to the dqds algorithm we refer to a simpler procedure
\emph{oqd} (orthogonal qd) that transforms $B$ into $\hat{B}$ so that
$$BB^T = \hat{B}^T\hat{B},$$ see~\cite{Fernando94}.

\fbox{
\begin{minipage}[t]{5.5cm}
\begin{algorithm}[oqd]
\label{alg:oqd}
\begin{center}
\begin{description}
\item $\tilde{a}_1 = a_1$
\item {\bf for} $k=1, \cdots,n-1$
\item \hspace{0.5cm} $\hat{a}_k:=\sqrt{\tilde{a}_k^2+b_k^2}$
\item \hspace{0.5cm} $\hat{b}_k:=b_k*(a_{k+1}/\hat{a}_k)$
\item \hspace{0.5cm} $\tilde{a}_{k+1}:=\tilde{a}_k*(a_{k+1}/\hat{a}_k)$
\item {\bf end for}
\item $\hat{a}_n := \tilde{a}_n$
\end{description}
\end{center}
\end{algorithm}
\end{minipage}
\quad
\begin{minipage}[t]{5.5cm}
\begin{algorithm}[dqds]
\label{alg:dqds}
\begin{description}
\item $d_1=q_1-s$,
\item {\bf for} $k=1,\cdots, n-1$
\item \hspace{0.5cm} $\hat{q}_k:=d_k+e_k$
\item \hspace{0.5cm} $\hat{e}_k:=e_k*(q_{k+1}/\hat{q}_k)$
\item \hspace{0.5cm} $d_{k+1}=d_k*(q_{k+1}/\hat{q}_k)-s$
\item {\bf end for}
\item $\hat{q}_n := d_n$
\end{description}
\end{algorithm}
\end{minipage} }

The dqd algorithm can be obtained by squaring the variables in the oqd algorithm,
\begin{equation}
\label{eq:qe}
q_k=a_k^2, \quad e_k=b_k^2,\quad  k=1,2,\cdots,n, \quad b_n = 0.
\end{equation}
The dqds algorithm is obtained by incorporating a shift in the dqd algorithm.
The shift will be denoted by \emph{s}.

The point of introducing oqd is to exhibit the relation of the
$\{q,e\}$ variables to the $\{a,b\}$ variables of $B$.  However an
alternative connection is to define
\[
L = \begin{bmatrix} 1 &&&&& \\ e_1 & 1 &&&& \\ & e_2 & . &&& \\ &&.& . && \\
    &&&. & 1 & \\ &&&&e_{n-1}&1  \end{bmatrix},
\quad
U= \begin{bmatrix}  q_1 & 1 &&&& \\ & q_2 & 1 &&& \\ & & . & . && \\
    & & & . & . & \\ & & & &. &1 \\ &&&&& q_n\end{bmatrix}.
\]
With some algebraic manipulations, one can show that the dqds algorithm implements the transform
\begin{equation}
\label{eq:ul1}
\hat{L} \hat{U} = UL - s I.
\end{equation}
One step of the dqds transform is to compute $\hat{L}$ and $\hat{U}$
from $L$ and $U$.

An essential ingredient in the success of dqds is that it
never explicitly forms any products $UL$ or $LU$ but works
entirely with the factors.
Furthermore, $LU$ and $B^TB$ are
connected by a diagonal similarity transformation~\cite{Parlett-Acta},
\begin{equation}
\label{eq:ulbtb}
LU=\Delta B^T B \Delta^{-1},
\end{equation}
where
\[
\Delta = \diag(1,\pi_1,\pi_1\pi_2,\ldots,\pi_1\cdots\pi_{n-1}), \quad \pi_i=a_ib_i.
\]
Therefore $LU$ and $B^TB$ have the same eigenvalues, and
$\sigma(B)=\sqrt{\lambda(LU)}$.

Recall that one step of the LR algorithm~\cite{BeresfordGutknech} transforms the matrix $A$ into
matrix $\hat{A}$ by
%by the following steps:
\[
\begin{split}
A-s I & = \tilde{L}\tilde{U}, \\
\hat{A} & = \tilde{U}\tilde{L}+s I,
\end{split}
\]
where $\tilde{L}$ is a lower triangular matrix and $\tilde{U}$ is an upper triangular matrix.
Note that $\hat{A}=\tilde{L}^{-1}A\tilde{L}$. We say that the shift is restored.
In contrast to the LR algorithm, the dqds transform is a non-restoring similarity transformation on
$UL$,
\begin{equation}
\label{eq:ul2}
\hat{U}\hat{L} = \hat{L}^{-1}(\hat{L}\hat{U})\hat{L} = \hat{L}^{-1}(UL - s I)\hat{L}.
\end{equation}
Notice that $\hat{U}\hat{L}$ is not similar to $UL$ and the shift $s$
is not restored. After each dqds transform, all eigenvalues are
reduced by $s$.  Because of the non-restoring feature, the dqds
algorithm must keep track of the accumulated shifts, $ S =
\sum_{i=1}^k s_i$.  %whenever the transformations shift the spectrum.

The dqds algorithm is the repeated applications of the dqds transform
with a sequence of well chosen shifts.  This algorithm checks, before
each transform, whether $e_{n-2}$ or $e_{n-1}$ is negligible.  If so
then one or two eigenvalues have been computed accurately and the
computation resumes on a smaller array. This process is called
deflation.  A well known feature of the QR-type and LR-type algorithms
is that the larger entries migrate gradually to the top of the matrix
and the smaller ones precipitate to the bottom.  The smaller they are
the quicker they fall.  That is why the classical deflation strategies
detect the bottom $2\times 2$ submatrix.

Detecting small entries above the bottom is also very important.
Setting a negligible $e_k$ to zero where $k < n-2$ is usually called
\emph{splitting}~\cite{Parlett-Imp} instead of deflation.  In this
case, a big problem breaks into two independent smaller problems.  It
is vital to check for splitting since, if $e_k=0$, the two subarrays
should be treated independently; whereas a tiny $e_k\neq 0$ could
signficantly slow down convergence if left untreated. The criteria for
negligibility have received careful attention and the tests can be
complicated. A novel deflation strategy is proposed in
section~\ref{sec:dmin} for the case when an intermediate variable
($e_k$ or $d_k$ for $k<n-2$) is negligible.

%The shift for the dqds algorithm is more strict than for the QR algorithm.
In the 1960s Kahan proved that a bidiagonal matrix defines
all its singular values to high relative accuracy, see~\cite{DK}.
In the context of dqds this precious
feature will be preserved if all the $\{q,e\}$ variables remain
positive~\cite{Fernando94}. Thus the shift $s$ must obey
\begin{equation}
\label{eq:dqds-shift}
s < \lambda_{\min}(BB^T).
\end{equation}
This constraint means that the singular values must be found in
monotone increasing order.  This is in contrast to the QR algorithm
with shifts but the QR algorithm does not claim to compute the
eigenvalues to high relative accuracy except when the shifts are zero.

The $\{q,e\}$ variables require so little storage that it is sensible
to compute $\{\hat{q},\hat{e}\}$ in a separate location from
$\{q,e\}$, and then decide whether to accept or reject the transform.
If any new variables are negative the algorithm rejects the transform
and chooses a smaller shift.  This is called a \emph{failure}. The
ability to reject a transform permits an aggressive shift strategy.

\subsection{Some theoretical results}

In the following sections we use the $B$ notation to describe
the dqds algorithm.

\begin{theorem}[Theorem 2 in \cite{Fernando94}]
\label{thm:dqd}
Apply the dqd transform (not dqds) to a positive bidiagonal matrix $B$
and $d_1,\cdots,d_{n}$ are the intermediate values. Then
\begin{enumerate}
\item $\sigma_{\min}^2(\hat{B}) \leq \min_k\{d_k\}$,
\item $[(BB^T)^{-1}]_{k,k}=d_k^{-1}$,
\item $(\sum_{k=1}^{n}d_k^{-1})^{-1} \leq \sigma_{\min}^2$.
\end{enumerate}
\end{theorem}

From item 1
$d_{\min} = \min_k\{d_k\} \ge \sigma_{\min}^2(\hat{B})$.
Asymptotically $d_{\min}$ is a very good estimate of
$\sigma_{\min}^2(\hat{B})$, and therefore
DLASQ records $d_{\min}$ as a guide to choosing a shift.
But at the early stage $d_{\min}$ may be too
big. There is more on this topic later.
\begin{corollary}
\label{cor:ming}
With the notation in Theorem~\ref{thm:dqd}, we have
\begin{equation}\label{eq:dmin1}
\frac{1}{n} d_{\min} \leq \sigma_{\min}^2(\hat{B}) \leq d_{\min} \leq n \sigma_{\min}^2(\hat{B}).
\end{equation}
\end{corollary}
Corollary~\ref{cor:ming} implies that $d_{\min}$ becomes negligible when the matrix becomes nearly singular.
It is the basis of our deflation strategy in section
\ref{sec:dmin}.
There are similar results for the dqds algorithm.

\begin{theorem}[Theorem 3 in \cite{Fernando94}]
\label{thm:bup2}
If the dqds with shift $s \ge 0$ transforms positive
bidiagonal matrix $B$ into positive $\hat{B}$ with
intermediate quantities $d_1,\cdots,d_{n}$ then
\begin{enumerate}
\item $\sigma_{\min}^2(\hat{B}) \leq \min_k\{d_k\}$,
\item $[(BB^T)^{-1}]_{k,k} < d_k^{-1}$,
\item $(\sum_{k=1}^{n}d_k^{-1})^{-1} < \sigma_{\min}^2(\hat{B})$.
\end{enumerate}
\end{theorem}
\begin{corollary}
\label{cor:ming2}
\begin{equation}
\label{eq:dmin}
\begin{split}
\frac{1}{n} d_{\min} <  & \sigma_{\min}^2(\hat{B}) \leq d_{\min} < n \sigma_{\min}^2(\hat{B}).
\end{split}
\end{equation}
\end{corollary}

These two corollaries have very important practical
implications.
\begin{itemize}
\item  A good shift should
  always be in $[\frac{d_{\min}}{n}, d_{\min}]$.
\item A very small $\sigma_n^2(\hat{B})$ automatically
  implies a very small $d_{\min}$.
  If $d_{\min}$ is negligible then it is possible to deflate the matrix even if $d_{\min}=d_k$, $k\ll n$, see section~\ref{sec:dmin}.
\end{itemize}

\section{Features of the current LAPACK implementation}
\label{sec:lapack}

The algorithm presented in~\cite{Parlett-Imp} was designed for speed
on all the test matrices available before 2000.  It was on average 5x
faster than the Demmel-Kahan QR code for computing singular
values~\cite{Parlett-Imp}.  The drive for efficiency produced a more
complex shift strategy than that used in the previous LR/QR codes for
the tridiagonal eigenvalue problems.  As an example, by unrolling the
last 3 minor steps, the code keeps values $d_{n-2}, d_{n-1}$ and $d_n$
available so that, in the asymptotic regime, there is enough
information to make a good approximation to the new $\sigma_{\min}$
even after one or two eigenvalues are deflated. For example, $d_{n-2}$
might be $d_{\min}$ for the deflated array.

The code keeps the information for both $\{q, e\}$ and $\{\hat{q},
\hat{e}\}$.  If any entry in $\hat{q}$ or $\hat{e}$ is negative the
transform is rejected and the step is considered a failure and aborted.
Consequently an aggressive shift strategy may be employed. But the
number of failures should be kept small. Too many failures will surely
degrade the performance.  To enhance data locality, the variables
are held in one linear array
\[ Z=\{q_1, \hat{q}_1, e_1, \hat{e}_1, q_2, \hat{q}_2 ,e_2, \hat{e}_2,
\cdots, q_n, \hat{q}_n \}.  \]
The dqds maps $\{q,e\}$ to $\{\hat{q},\hat{e}\}$ and  vice versa
to avoid data movement.  This is called \emph{ping-pong}
implementation in~\cite{Parlett-Imp}.

The implementation of dqds in the LAPACK, as of 2000, is based on the
following perception.
\begin{quote}
``An explicit conditional statement (if-then-else) in an inner loop
impedes efficient implementation on a pipelined arithmetic unit."
\end{quote}

Consequently the division in the dqds loop, $t=q_{k+1}/\hat{q}_k$, is
not protected from incurring an exception (divide by zero or
overflow). However, the powerful feature of arithmetic units
conforming to IEEE floating point standard 754 is that computation is
not held up by an exception. At the end of the loop the code tests
whether an $\infty$ or a NaN (not a number) occurred and acts
appropriately.

The LAPACK implementation also assumes that a good compiler will
implement the intrinsic FORTRAN functions such as MIN(A,B) or MAX(A,B)
efficiently. Hence the valuable variable $d_{\min}$ is computed via
$d_{\min}$= \text{MIN}($d_k, d_{\min}$) in the inner loop.

This implementation also gives up the knowledge of the index
at which $d_{\min}$ receives its final value since this would
require an explicit conditional,
\begin{enumerate}[\hspace{0.5cm} ]
\item if $d_k< d_{\min}$ then
\item \hspace{0.5cm} dmink=$k$;  $d_{\min}=d_k$;
%\item \hspace{0.5cm}
\item end if
\end{enumerate}
where the position of $d_{\min}$ is denoted by $dmink$. However, our
efficiency enhancements require knowledge of the index of $d_{\min}$.
This explicit conditional statement need not impede performance
provided that it is placed, not in its natural position, but
immediately after the division.  This is because division is so slow
relative to other operations, such as comparison, that the conditional
statement can be completed before the preceding division finishes.
Technically this requires us to update $d_{\min}$ one minor step late
but that is easily dealt with.  Algorithm~\ref{alg:innerloop} is the
inner loop of a new implementation.

\fbox{
\begin{minipage}[t]{12cm}
\begin{algorithm}[inner loop of the dqds algorithm]
\label{alg:innerloop}
\begin{description}
\item \hspace{0.5cm} {\bf for} $k=1$ to $n-1$
\item \hspace{1.0cm} $\hat{q}_k =d_k+e_k$; $t = q_{k+1}/\hat{q}_k$;
\item \hspace{1.0cm} {\bf if} $d_k < 0$ {\bf then}
\item \hspace{1.5cm} exit \emph{(early failure)};
\item \hspace{1.0cm} {\bf else}
\item \hspace{1.5cm} {\bf if} $d_k < d_{\min}$ {\bf then}
\item \hspace{2.0cm} dmink=$k$; $d_{\min}=d_k$;
\item \hspace{1.5cm} {\bf end if}
\item \hspace{1.0cm} {\bf end if}
\item \hspace{1.0cm} $\hat{e}_k = e_k\cdotp t$; $d_{k+1}=d_k\cdotp t - s$;
\item \hspace{0.5cm} {\bf end for}
\item \hspace{0.5cm} {\bf if} $d_n < 0$ {\bf then}
\item \hspace{1.0cm} continue \emph{(late failure)};
\item \hspace{0.5cm} {\bf else if} $d_n < d_{\min}$ {\bf then}
\item \hspace{1.0cm} dmink=$n$; $d_{\min}=d_n$;
\item \hspace{0.5cm} {\bf end if}
\end{description}
\end{algorithm}
\end{minipage} }

If $d_k < 0$ ($k<n$) and shift $s>0$, then by Algorithm~\ref{alg:dqds}
the dqds algorithm may fail in three possible ways: $\hat{q}_j$ ($k< j
<n$) is negative or zero, or $\hat{q}_n<0$.  The case ($d_k<0, k<n$)
is called an \emph{early failure}.  A \emph{late failure} occurs when
the arrays $\{\hat{q}_k,\hat{e}_k\}$ are all positive except for the
last $\hat{q}_n=d_n<0$.  In the case of a late failure, a smaller
shift can be chosen as $s+d_n$, which is guaranteed to succeed.  This
property was discovered by H. Rutishauser. See~\cite{Parlett-Imp} for
details.

\section{Improvements}
\label{sec:improvement}

In this section, we summarize the improvements of our implementation
over DLASQ, which can be divided into two types:
\emph{deflation strategies} and \emph{shift strategies}.

\subsection{Setting negligible $d_{\min}$ to zero}
\label{sec:dmin}

Recall that the dqds algorithm is non-restoring. The algorithm tries
at every step to make the current matrix singular.  The
positive $\{q,e\}$ array defines both a positive bidiagonal matrix $B$
and the matrices $L$, $U$, see~\cite{Parlett-Acta}. It can happen that
a leading principal submatrix of $B$ becomes almost singular long
before any negligible entries appear at the bottom of the matrix.
This situation is not easily detected by a simple inspection of the
entries $\{q_k,e_k\}$, or say $\{a_k,b_k\}$.

To explain, it is best to go back to the \emph{oqd} transform from
$B^T$ to $\hat{B}$, and consider the process after ($k-1$)
minor steps shown in the following equation,
\begin{equation}
\label{eq:singleton}
B^{(k)} = Q_k B^T = \begin{bmatrix}
\hat{a}_1 & \hat{b}_1 &&&&&&&& \\
0 & \hat{a}_2 & \hat{b}_2 &&&&&&& \\
& 0 & \cdotp & \cdotp &&&&&& \\
&&0 & \hat{a}_{k-1} & \hat{b}_{k-1} &&&&& \\
&&&0 & \tilde{a}_k & 0 &&&& \\
&&&& b_k & a_{k+1} & 0 &&&\\
&&&&& b_{k+1} & a_{k+2} & & \\
&&&&&& \cdotp & \cdotp & 0 & \\
&&&&&&& b_{n-1} & a_n
\end{bmatrix}.
\end{equation}

The striking feature is that row $k$ is a singleton, its entry is
$\tilde{a}_k$ and $\tilde{a}_k^2 = d_k$.  Rows ($1:k$) are upper
bidiagonal, rows ($k:n$) are lower bidiagonal.  Such matrices are
often described as \emph{twisted}.  Corollary~\ref{cor:ming} says
$\tilde{a}_k$ would be small when $\hat{B}$ is nearly singular.  How
small must $\tilde{a}_k$ be to declare it negligible?  Define new
matrices $\tilde{B}$ and $E$, using equation~\eqref{eq:singleton}, by
\[
B^{(k)}=Q_k B^T = \tilde{B}+E,
\]
where row $k$ of $\tilde{B}$ is null and
$E=\text{diag}(0,\cdots,0,\tilde{a}_k,0,\cdots,0)$.  Observe that
$E\tilde{B}=E^T\tilde{B}=0$. Hence,
\begin{equation}
\label{eq:e2}
\begin{split}
B^{(k)T}B^{(k)} & =(\tilde{B}+E)^T(\tilde{B}+E) \\
& = \tilde{B}^T\tilde{B}+\tilde{B}^TE+E^T\tilde{B}+E^TE \\
& = \tilde{B}^T\tilde{B} + E^TE.
\end{split}
\end{equation}

By Weyl's monotonicity theorem~\cite[Thm. 3.3.16]{Horn-book2}), for $i=1:n$,
\begin{equation}
|\sigma_i^2-\tilde{\sigma}_i^2| \le \|E^2\|=d_k.
\end{equation}
Here $\{\tilde{\sigma}_i^2\}$ are the ordered singular values of
$\tilde{B}$.  The non-restoring character of the dqds algorithm
entails that the desired eigenvalues are $\{\sigma_i^2+S\}$ where $S$
is the accumulated sum of shifts so far.  Consequently, $d_k$ (%and
$\tilde{a}_k^2$) may be set to zero, when
%%\begin{equation}
\[ d_k \le \epsilon S \le \epsilon (S+\sigma_i^2).  \]
%\end{equation}

How to deflate by exploiting such an event?
% Some more work must be done.
Examination of the inner loop of {\em oqd}, Algorithm~\ref{alg:oqd},
shows that with $\tilde{a}_k=0$, the algorithm simply moves the
remaining variables into new positions
\[
\hat{a}_j=b_j, \text{ }
\hat{b}_j=a_{j+1}, \text{ } j=k,\ldots, n-1, \text{ and  } \hat{a}_n=\tilde{a}_k=0.
\]
The bidiagonal $\hat{B}$ reveals its singularity ($\hat{a}_n = 0$) but
deflation requires that $\hat{b}_{n-1}$ also vanishes.  More work needs
to be done since $\hat{b}_{n-1}=a_n$ may not be negligible.  There are
two options.

{\bf Option A.} Apply the {\em oqd} transform to $\hat{B}$ to obtain
$\bar{B}$ and note that
$\bar{b}_{n-1}=\hat{b}_{n-1}(\hat{a}_n/\bar{a}_{n-1})=0$, $\bar{a}_n =
\tilde{a}_{n-1}(\hat{a}_n/\bar{a}_{n-1})=0$.  The new singular value
is $S+\hat{a}_n =S$ and $n \leftarrow n-1$.  In practice the dqd
transform is used, not oqd.

{\bf Option B.} The procedure invoked in the aggressive early
deflation algorithm, see~\cite{dqds-Agg}, is useful here.
In our implementation we use this option since we found it
usually saves some floating point operations over {\bf Option A}.
We describe it briefly.

Apply a carefully chosen sequence of plane rotations on the
right of $\hat{B}$ to chase the entry $\hat{b}_{n-1}$ up in the
last column of $\hat{B}$.  The sequence of `planes' is
$(n-1,n)$, $(n-2,n),(n-3,n),\cdots,(1,n)$.  The single
nonzero entry $\psi$ in the last column is called the bulge.  Its
initial value is $\hat{b}_{n-1}$ and it shrinks as it rises
up.  The expectation is that the bulge will
become negligible quickly, in fewer than 10 rotations.

Let ${\bf e}_i$, $i=1,\ldots, n$, be the $i$-th column of
an $n\times n$ identity matrix (to distinguish from $e_i$).
When the bulge is in position $(k,n), k\le
n-1$, the matrix can be written as $\overset{o}{B}+E$
where $E=\psi {\bf e}_k {\bf e}_n^T$ and $\overset{o}{B}$'s last row
and column are null.  Then
$(\overset{o}{B}+E)(\overset{o}{B}+E)^T=\overset{o}{B}\overset{o}{B^T}+\psi^2
{\bf e}_k{\bf e}_k^T$.
It turns out that the criterion
for neglecting the bulge $\psi$ is the same as the
criterion for neglecting $d_k$, namely
\[
\psi^2 \le \epsilon S,
\]
and $\psi^2$ is expressible in terms of $\{q_i,e_i\}$.

To be more specific, let $\psi = \sqrt{x}$. The bulge
moves upward by one position by applying a Givens transformation from the right~\cite{dqds-Agg}:
\[
\begin{bmatrix} * & * &  && \\ & * & * &&  \\
&& * & \sqrt{e_k} &  \\ &&& \sqrt{q_{k+1}} & \sqrt{x} \\&&&&0 \end{bmatrix} \underrightarrow{G_k}
\begin{bmatrix} * & * &&& \\ & * & * &&  \\ &&  & \sqrt{\bar{e}_k} & \sqrt{\bar{x}} \\ &&&\sqrt{\bar{q}_{k+1}}& 0 \\&&&&0 \end{bmatrix},
\]
where $G_k$ is an orthogonal matrix of the form
$\begin{bmatrix} c & s \\ s & -c \end{bmatrix}$ to keep the entries of $B$ positive.
The values $\bar{q}_{k+1}$, $\bar{e}_k$ and $\bar{x}$ can be computed as
\[
\bar{q}_{k+1} = q_{k+1}+x, \quad \bar{e}_k = \frac{q_{k+1}e_k}{q_{k+1}+x}, \quad
\bar{x} = \frac{x e_k}{q_{k+1}+x}.
\]
The whole algorithm is described as follows, where
$\bar{q}, \bar{e}$ and $\bar{x}$ are also denoted by $q, e$ and $x$ respectively.

\fbox{
\begin{minipage}[t]{12cm}
\begin{algorithm}[{\bf Option B}]
\begin{description}
\item \hspace{0.5cm} $x=e_{n-1}$;
\item \hspace{0.5cm} {\bf for} $k=n$-$1,n$-$2,\cdots, 2$
\item \hspace{1.0cm}  $t_1 = q_k$;
\item \hspace{1.0cm} $q_k=q_k+x$;
\item \hspace{1.0cm} $t_2=1/q_k$;
\item \hspace{1.0cm} $x= x\cdotp e_{k-1} \cdotp t_2$;
\item \hspace{1.0cm} $e_{k-1}=e_{k-1}\cdotp t_1 \cdotp t_2$;
\item \hspace{1.0cm} {\bf if} $x$ is negligible, {\bf break};
\item \hspace{0.5cm} {\bf end for}
\item \hspace{0.5cm} {\bf if} $k=1$, \quad $q_1=q_1+x$;
\end{description}
\label{alg:def3}
\end{algorithm}
\end{minipage}}

\begin{remark}
\label{rmk:zerodef}
The deflation strategy, setting negligible $d_{\min}$ to zero, is the
most important improvement over DLASQ and is `tailor-made' for the
disordered matrices for which the dqds algorithm shows slow
convergence.  This deflation will be called \emph{d-deflation
strategy} in later sections.
\end{remark}

\subsection{Improved criterion for late deflation}
\label{sec:newdefs}

A crude, but adequate, criterion for setting $b_{n-1}$ (in $B$) to
zero, is
\[
|b_{n-1}| < c \epsilon (\sqrt{S}+\sigma_{{\min}}(B)),
\]
where $\epsilon$ is the machine precision and $c$ is a modest
constant, for example $c=10$ in DLASQ and in our implementation.  In the
context of a $\{q, e\}$ array this same criterion becomes
approximately
\begin{equation}
\label{eq:crudecri}
e_{n-1} < (c\epsilon)^2 S.
%\le (c\epsilon)^2 (\sqrt{S}+\sigma_{\min}(B))^2.
\end{equation}
%This is what is used in DLASQ.

Note that $q_n=d_n$ usually becomes very small,
see Lemma 5.3~\cite{Aishima-Convergence}
or Lemma~\ref{lem:prov2} in the Appendix.
Instead of only testing $e_{n-1}$, a more refined criterion
arises from considering the trailing $2\times 2$ submatrix
of $BB^T$,
\[
\begin{bmatrix} a_{n-1}^2+b_{n-1}^2 & b_{n-1}a_n \\ b_{n-1}a_n & a_n^2 \end{bmatrix}.
\]
By Weyl's theorem no eigenvalue of $BB^T$ changes by no more than
$b_{n-1}a_n$ if the $2\times 2$ matrix $\begin{bmatrix} 0 & b_{n-1}a_n
\\ b_{n-1}a_n & 0 \end{bmatrix}$ is subtracted from the trailing
$2\times 2$ submatrix shown above.  However the $(n-1,n-1)$ entry of
$BB^T$ still involves $b_{n-1}^2$.  No eigenvalue of $BB^T$ can change
by more than $b_{n-1}^2$ if it is neglected.  We can set it to zero
when it is negligible compared to either $S+\lambda_{\min}(BB^T)$ or
$a_{n-1}^2$.  Consequently the crude criterion~\eqref{eq:crudecri} may
be replaced by the following pair of tests:
\[
b_{n-1}^2 < c\epsilon \max (S,a_{n-1}^2) \text{ and } b_{n-1}a_n < c\epsilon(S+ \lambda_{\min}(BB^T).
\]

In the context of $\{q, e\}$ array, we use
\begin{equation}
\label{eq:parlett2}
e_{n-1} < c \epsilon \max(S,q_{n-1}) \text{ and } e_{n-1}q_n < (c\epsilon S)^2.
\end{equation}

The same arguments can give us a more refined test for splitting $B$
whenever $e_k$ and $e_kq_{k+1}$ are negligible.

\subsection{Updating the upper bound}

In his original papers on the qd algorithm, Rutishauser proposed
updating the upper and lower bounds, {\em sup} and {\em inf}, on
$\sigma_{\min}^2(B)$ at all times.  The original
paper~\cite{Fernando94} on the dqds algorithm followed the
recommendation but the implementation DLASQ~\cite{Parlett-Imp} omitted
to update {\em sup} when a transform failed.

A transform $\{q, e\}$ into $\{ \hat{q}, \hat{e}\}$ fails if any entry
in $\hat{q}$ or $\hat{e}$ is nonpositive (but $\hat{q}_n=0$ is
permitted). In this case the connection to a bidiagonal $\hat{B}$ is
lost. It can only happen if the shift $s$ exceeds $\sigma_{\min}(B)^2$
and thus {\em sup} can be set to $s$. This is valuable information.

Here is the pseudo-code, with $s$ being the shift.

\fbox{
\begin{minipage}[t]{12cm}
\begin{algorithm}[updating the upper bound] \vspace{0.1cm}
\begin{description}
\item \hspace{0.5cm} {\bf if} shift $s$ succeeds, {\bf then} \vspace{0.1cm}
\item \hspace{1.0cm} $sup$=$\min$\{$d_{\min}$, $sup$-$s$\}, \vspace{0.1cm}
\item \hspace{0.5cm} {\bf else if} shift $s$ fails, {\bf then} \vspace{0.1cm}
\item \hspace{1.0cm} $sup$=$\min$\{$s$, $sup$\},\vspace{0.1cm}
\item \hspace{0.5cm} {\bf end if}
\end{description}
\end{algorithm}
\end{minipage} }

In case of failure the transform $\hat{q}$ and $\hat{e}$ is discarded.
The shift is usually computed as $s=\alpha \cdotp d_{\min}$ in DLASQ,
where $\alpha$ is a parameter in $(0, 1)$.  We replace DLASQ's
$d_{\min}$ by \emph{sup} and the reward is a reduction in the number
of failures.  As mentioned before, $d_{\min}$ is used to estimate
$\lambda_{\min}(BB^T)$.  At the early stage, $d_{\min}$ may be too
large and $sup$ may be a much better upper bound than $d_{\min}$.

The negligibility of \emph{sup} or $d_{\min}$ reveals the convergence
of a shifted singular value to 0. DLASQ does not check its
negligibility, and so computes a complicated shift when shift = 0
would suffice.

\subsection{Twisted shift in the last $p$ rows}
\label{sec:twist}

When $d_{\min}$ is near the bottom, we use the twisted factorization
to choose a shift.  We call this shift strategy \emph{restricted
twisted shift strategy} since it is only used when $d_{\min}$ is in
the last $p$ rows. Our experiments suggest $p = 20$ is a good
choice. A similar strategy is used in DLASQ (Cases 4 and 5), when
$d_{\min}$ is in the last two rows.

The idea behind this shift strategy is to compute an approximate
eigenvector $z$ of $BB^T$ by focusing on the submatrix around
$d_{\min}$.  Using $z$ and the Rayleigh quotient residual, we can
compute a lower bound $\phi$ on the smallest eigenvalue with high
likelihood that $\phi < \sigma_{\min}(B)$, and $\phi$ can be used as a
shift, see section 6.3.3
of~\cite{Parlett-Imp} or the Appendix for details.

With the techniques similar to those in~\cite{Aishima-Convergence}, we
can show that the order of convergence for the twisted shift strategy
is $1.5$, see the following theorem.

\begin{theorem}
\label{thm:mainconv}
Assume the bidiagonal matrix $B$ has positive nonzero entries.
For the dqds algorithm with the twisted shift strategy, the
sequence $\{e_{n-1}^{(l)}\}_{l = 0}^{\infty}$ converges to
$0$ with order of convergence $1.5$.
\end{theorem}

We leave the proof of Theorem~\ref{thm:mainconv} to the
Appendix.

\subsection{Using a suitable 2-by-2}
\label{sec:kahan}

Kahan \cite{Kahan-idea} suggests a better method to compute an upper
bound that requires little overhead. The idea is that the smallest
singular value of the submatrix around $d_{\min}$ must be a better
upper bound than $d_{\min}$ itself.

Let us consider the odq algorithm first.  Let $k$ be the index where
$\tilde{a}_k=\min_i(\tilde{a}_i)$, and assume $k>1$, see
equation~\eqref{eq:singleton}.  By the interlacing property of the
singular values~\cite{Golub-book2}, we know the smallest singular
value of any principle submatrix of $B^{(k)}$ is larger than
$\sigma_{\min}(B)$.  Thus the smallest singular value of
$\begin{bmatrix} \hat{a}_{k-1} & \hat{b}_{k-1} \\ & \tilde{a}_k
\end{bmatrix}$ is an upper bound on $\sigma_{\min}(B)$.  This claim is
also valid for the dqds algorithm.

\subsection{Our whole shift strategy}

In this subsection we summarize the structure of our shift strategy.
We modify the shifts of DLASQ~\cite{Parlett-Imp} in two places.
\begin{itemize}
\item The Case 4 and 5 in DLASQ use the twisted shift strategy when $d_{\min}$ is in the last $(p=)$ $2$ rows.
  We replace it by $p=20$.

\item We replace all $d_{\min}$ by \emph{sup}.
\end{itemize}

\fbox{
\begin{minipage}[t]{12cm}
\begin{algorithm}[Our shift strategy]
\begin{enumerate}[\hspace{0.5cm} (1)]
\item {\bf if} $d_{\min}<0$,
    \begin{itemize}
\item[] $s = 0$; \hspace{3.0cm} --Case 1
      \end{itemize}
\item {\bf if} no eigenvalue deflated
  \begin{itemize}
  \item {\bf if} ($d_{\min}$ in the last 20 rows)
    \begin{itemize}
    \item {\bf if} ($d_{\min}$ in the last two rows),
     \item[] \hspace{0.5cm} use the old shift strategy; \hspace{0.6cm} -- Case 2 and 3
    \item {\bf else}
    \item[] \hspace{0.5cm} use the \emph{twisted shift strategy}; \hspace{0.1cm} -- Case 4 and 5
    \item {\bf end if}
      \end{itemize}
  \item {\bf else}
    \begin{itemize}
    \item update $sup$ via the technique in section~\ref{sec:kahan};
    \item  use the old strategy and replace $d_{\min}$ by \emph{sup};   \hspace{0.6cm} -- Case 6
    \end{itemize}
    \end{itemize}
\item {\bf else}
    \begin{itemize}
  \item[] use the old shift strategy.
      \end{itemize}
\item {\bf end if}
\end{enumerate}
\end{algorithm}
\end{minipage} }
As in DLASQ, $d_{\min}<0$ is used to flag a new segment after splitting~\cite{Parlett-Imp}.
See section 6.3 in~\cite{Parlett-Imp} for the specific definitions of different cases.

\section{Finite step convergence property}
\label{sec:finiteCov}

In this section we justify our claim of linear worst case complexity
of our improved algorithm.
%, the prototype of which is introduced in Algorithm~\ref{alg:proto}.
This property has allowed us to implement
our algorithm without fear of reaching an iteration limit
before convergence occurs.

Algorithm~\ref{alg:proto} shows the prototype of our algorithm.
The main differences with DLASQ are that Algorithm~\ref{alg:proto} updates
the upper bound $sup$ while DLASQ does not, and that
Algorithm~\ref{alg:proto} uses the d-deflation strategy while DLASQ does
not.  For simplicity, other techniques are not included in
Algorithm~\ref{alg:proto} such as \emph{splitting} and \emph{flipping},
see~\cite{Parlett-Imp} for details.
One step of dqds transform which either succeeds or fails is called
 one \emph{dqds iteration}, or just one \emph{iteration}.
%For easy analysis, we assume every shift is computed as a fraction of $sup$.
% \begin{theorem}
% \label{thm:finiteCov}

 Our claim is as follows: For an $n \times n$ bidiagonal matrix $B$ with positive elements, Algorithm~\ref{alg:proto}
 can compute each singular value in about
 \begin{equation}
 \label{eq:max_it}
 % k=\lceil \frac{\log(n/\epsilon)}{\log (1/\beta)}\rceil
 \Upsilon=\lceil \log_{1/\beta}(n/\epsilon) \rceil
\end{equation}
iterations, where $\beta=\max(\alpha,1-\alpha)$
and $\epsilon$ is the machine precision, and in total it requires about $O(n\Upsilon)$
iterations to find all singular values.
%\end{theorem}

We first do one dqd transform on matrix $B$ and then initialize $sup = d_{\min}$.
By Corollary~\ref{cor:ming} or~\ref{cor:ming2}, we know $sup
\le n \sigma_{\min}^2(B)$.
Denote the upper bound after the $k$th dqds iteration by $sup^{(k)}$.
Since $sup^{(k)}$ is used as a guide to choosing shift and
each iteration would make it
smaller at least by $(1-\beta) \cdotp sup^{(k)}$,
the upper bound after $k+1$ iterations satisfies
\[
sup^{(k+1)} \leq \beta sup^{(k)} \leq \beta^k n \sigma_{\min}^2(B),
\]
where $\beta=\max(\alpha,1-\alpha)$ and $\alpha$ is defined as in Algorithm~\ref{alg:proto}.

Thus, for
% $k \geq \frac{\log(n/\epsilon)}{\log (1/\beta)}$,
$k \geq \log_{1/\beta}(n/\epsilon)$,
we would have
\[
sup^{(k+1)} \le \beta^k n \sigma_{\min}^2(B) \le \epsilon \sigma_{\min}^2(B).
\]
If $k\ge \Upsilon$ and matrix $B$ is still not deflated by classical
strategies, by
Corollary~\ref{cor:ming} or~\ref{cor:ming2}, $B$ must be deflated
by the d-deflation strategy in section~\ref{sec:dmin}.
Therefore, Algorithm~\ref{alg:proto} requires no more than $\Upsilon$ iterations per singular value.

\fbox{
\begin{minipage}[t]{12cm}
\begin{algorithm}[Prototype of our algorithm]
\label{alg:proto}
\begin{description}
    \item \hspace{0.5cm} Choose a parameter $\alpha \in (0 \;\; 1)$;
    \item \hspace{0.5cm} {\bf while} Z unfinished do
    \item \hspace{1.0cm} (1) deflate array Z
    \item \hspace{1.5cm} {\bf if}  Z's final entries are negligible {\bf then}
    \item \hspace{2.0cm} reduce Z accordingly;
    \item \hspace{1.5cm} {\bf else}
    \item \hspace{2.0cm} {\bf if } $d_{\min}$ is small enough, use d-deflation strategy;
    \item \hspace{1.5cm}{\bf end if}
    \item \hspace{1.0cm} (2) choose a shift
    \item \hspace{1.5cm} {\bf if }$sup$ is small enough, {\bf then}
    \item \hspace{2.0cm} choose $s=0$;
    \item \hspace{1.5cm} {\bf else}
    \item \hspace{2.0cm} $s=\alpha\cdotp sup$;
    \item \hspace{1.5cm} {\bf end if}
    \item \hspace{1.0cm} (3) apply a dqds transform to Z
    \item \hspace{1.5cm} update $sup$;
    \item \hspace{1.5cm} {\bf if} dqds fails, {\bf go to} step (2);
    \item \hspace{0.5cm} {\bf end while}
\end{description}
\end{algorithm}
\end{minipage} }

If $\alpha = \frac{3}{4}$, $\epsilon = 10^{-16}$ and $n=1000$, the
algorithm above requires no more than $152$ iterations to compute each
singular value.
%Theorem~\ref{thm:finiteCov}
%shows the importance of d-deflation strategy for a robust
%implementation of the dqds algorithm, and also
Equation~\eqref{eq:max_it} gives a guide to setting the average iteration number
required for all singular values.
It is set to 30 in DLASQ (LAPACK-3.4.0) which is too small,
and DLASQ may fail for the difficult matrices in section~\ref{sec:Num-dqds}.
Equation~\eqref{eq:max_it} suggests that a number around 100 is a better choice.
%We emphasize that Theorem~\ref{thm:finiteCov} is worst case analysis.
%and shows the importance of d-deflation for a robust implementation of the dqds algorithm.
In the absolute majority of cases, our algorithm converges very fast, in the range of
about $10$ iterations per singular value.  However, there are some
bidiagonal matrices for which it indeed requires $O( \Upsilon)$ iterations to find
the smallest singular value, e.g. the difficult matrices in
section~\ref{sec:Num-dqds}. After the smallest singular
value is found convergence is very quick. %where $\Upsilon$ is defined in~\eqref{eq:max_it}.
The worst case iteration bound is very similar to that in the
zero-in algorithm and its recent variant for finding zeros of a
general univariate nonlinear equation (see~\cite{zeroin,Gu-zeroin}).

%Just as in DLASQ~\cite{Parlett-Imp}, our experiments show that
%the parameter $\alpha$ is better not to be a constant.
%For example, we can adaptively
%increase $\alpha$ if previous shifts succeed, or decrease $\alpha$ if
%previous shifts have failed consecutively for three times.

Just as in DLASQ~\cite{Parlett-Imp}, we in practice can gradually increase $\alpha$ if previous shifts
succeed and decrease $\alpha$ if previous shifts have failed continuously for three times.

\section{Numerical Results}
\label{sec:Num-dqds}

We have implemented our algorithm in Fortran 77 by incorporating our
new deflation and new shift strategies into the DLASQ routines.  All
experiments were performed on a laptop with 4G memory and Intel(R)
Core(TM) i7-2640M CPU.  For compilation we used the \texttt{gfortran}
compiler and the optimization flag \texttt{-O3}, and linked the codes
to optimized BLAS and LAPACK, which are obtained by using
ATLAS~\cite{Atlas}.

\subsection{Comparing each new technique} Recently, we have been able
to construct a number of bidiagonal matrices of different dimensions
for which DLASQ converges so slowly that it requires more iterations
than allowed in the code. In these matrices, the diagonals are highly
disordered and both the diagonal and off-diagonal entries are of
massively varying orders of magnitude.  Figure~\ref{fig:tmatrix} shows
the values of the diagonal elements of one such matrix, where the
x-axis denotes the index of the diagonal elements and the y-axis
denotes their values.  The plot of the first 800 diagonal elements
goes up and down, while the last 200 diagonal elements are nearly
equal to one.

In this section, we use these difficult matrices and a random
bidiagonal matrix whose entries are from Gaussian distribution to show
the improvements of our modifications.  There are five main
modifications over DLASQ.  Consequently we have five
versions of improved algorithms by adding the new techniques one by
one.  We use the following notation to denote them,
\begin{itemize}
\item V1: d-deflation strategy in section~\ref{sec:dmin};
\item V2: new deflation strategies in section~\ref{sec:newdefs};
\item V3: updating the upper bound;
\item V4: twisted shift in last 20 rows;
\item V5: using suitable 2-by-2 submatrix.
\end{itemize}

\begin{table}[ptb]
\caption{The difficult bidiagonal matrices}%
\label{tab:matrices}
\begin{center}%
\begin{tabular}
[c]{|cccc|}\hline
& $n$  & Description of the bidiagonal matrix $B$ & Source \\ \hline \hline
1 & 544 & Matrix\_1 &  \cite{MarquesWebsite} \\
2 & 1000 & Matrix\_2 & \cite{MarquesWebsite} \\
3 & 1087 & Matrix\_3 & \cite{MarquesWebsite} \\
4  & 1088 & Matrix\_4 & \cite{MarquesWebsite} \\
5 & 5000 & random bidiagonal matrix  &  \\ \hline
\end{tabular}
\end{center}
\end{table}

V$i$ denotes a version after adding a new technique to the previous
version.  Thus V5 is the best improved algorithm.  Note that the
technique \emph{setting negligible $d_{\min}$ to zero} has been
implemented in LAPACK-3.4.1\footnote{This technique is added to
LAPACK-3.4.0 by J. Demmel, W. Kahan and B. Lipshitz.}, and V1 performs
similarly to the DLASQ in LAPACK-3.4.1.  Table~\ref{tab:time-dqds}
shows the speedups of V5 over DLASQ in terms of time.  The average
iterations of V5 for each singular value are shown in
Table~\ref{tab:A2}.  From the tables we can see that V5 requires far
fewer iterations and is about 3x-10x faster than DLASQ for these
difficult matrices.

% \begin{table}[ptb]%
% \centering
% \caption{The flops of DLASQ over the new one \label{tab:div1}}{%
% \begin{tabular}{|c|c|cccccc|}
% \hline
% Matrix  &  $n$     &   DLASQ &    M1  &   M2   &   M3   &     M4    &    M5     \\ \hline \hline
% 1       & 544    &  1.00  &  1.93  &  2.00 &  2.94  &  3.16  &   3.83    \\
% 2       & 1000  &  1.00  &  1.84  &  1.91 &  2.68  &  2.82  &   3.30     \\
% 3       & 1087   &  1.00  &  5.22  &  5.68 &  9.11  &  9.47 &   10.49   \\
% 4       & 1088   &  1.00  &  2.00  &  2.12 &  3.03  &  3.43 &   3.81   \\
% 5       & 5000   &  1.00  &  1.05  &  1.05 &  1.08  &  1.07 &   1.16  \\
% \hline
% \end{tabular}}
% \end{table}

\begin{table}[ptb]%
\centering
\caption{The speedup of V$i$ over DLASQ \label{tab:time-dqds}}{%
\begin{tabular}{|c|c|cccccc|}
\hline
Matrix & $n$  &  DLASQ &    V1  &   V2   &   V3     &   V4   &   V5    \\ \hline \hline
1      & 544  & 1.00  &  1.62 &  1.75 &  2.14   &  2.10 &  2.83     \\
2      & 1000 & 1.00  &  1.76 &  1.76 &  1.80   &  2.06 &  2.74   \\
3      & 1087 & 1.00  &  4.98 &  5.20 &  7.88   &  7.88 &  9.99    \\
4      & 1088 & 1.00  &  1.79 &  2.14 &  2.41   &  2.52 &  3.31     \\
5      & 5000 & 1.00  &  1.04 &  1.04 &  1.07   &  1.14 &  1.20    \\
\hline
\end{tabular}}
\end{table}

\begin{table}[ptb]
\caption{The average iterations for each singular value}%
\label{tab:A2}
\begin{center}%
\begin{tabular}
[c]{|c|ccccc|}\hline
Matrix & 1  & 2 & 3 & 4 & 5  \\ \hline \hline
DLASQ  & 45.4 & 26.0 & 74.8 & 30.7 & 10.5  \\
V5 & 11.81 & 7.19 & 7.62 & 8.85 & 7.78 \\ \hline
\end{tabular}
\end{center}
\end{table}

This superior performance of V5 is largely due to the d-deflation
strategy in section~\ref{sec:dmin}, which is almost `tailor-made' for
such difficult matrices.  It can deflate $40\%$-$80\%$ singular values
of such matrices (see Figure~\ref{fig:percent}.)  It is also very
interesting to note that it usually can deflate about $20\%$ of
singular values of a general bidiagonal matrix.

For the Matrix\_3, Figure~\ref{fig:first500} shows the locations of
negligible $d_{\min}$ when finding the first 500 singular values,
which are all deflated by the d-deflation strategy.  For this matrix, V5
achieves about 10x speedup.  We compare the number of iterations when
finding the first 10 singular values of Matrix$\_3$.  The results are
shown in Figure~\ref{fig:first10it}.  For the first 10 singular
values, DLASQ takes 120 iterations on average.

The last four improvements aim to accelerate the convergence for
general bidiagonal matrices. From Table~\ref{tab:time-dqds}, we can
see that each of these techniques improves the performance.  Since
different matrices may have very different properties, we can not
expect these techniques to help a lot for all matrices.  We further
use all the matrices in {\tt stetester}~\cite{A880} to test each
improvement by adding them one by one.  Each technique is helpful for
most of these matrices. Due to space limitations, we do not include
all these results.

 \begin{figure}[ptbh]
 \centering
 \subfigure[The values of diagonal elements of Matrix\_3]{
 \includegraphics[totalheight=2.20in,width=2.45in]{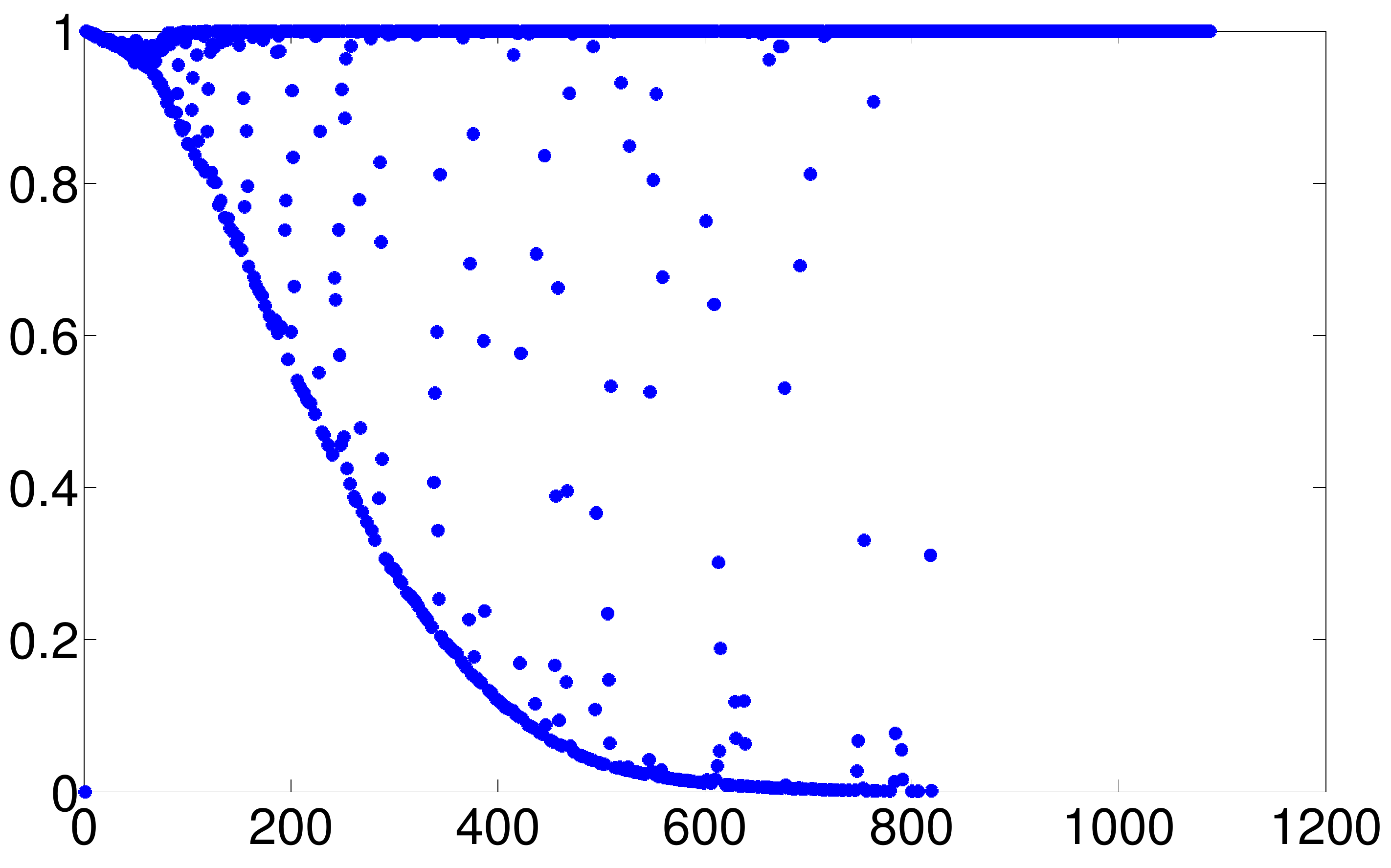}
 \label{fig:tmatrix}}
 \subfigure[The percentage of singular values deflated by d-deflation strategy in section~\ref{sec:dmin}]{
 \includegraphics[height=2.20in,width=2.45in]{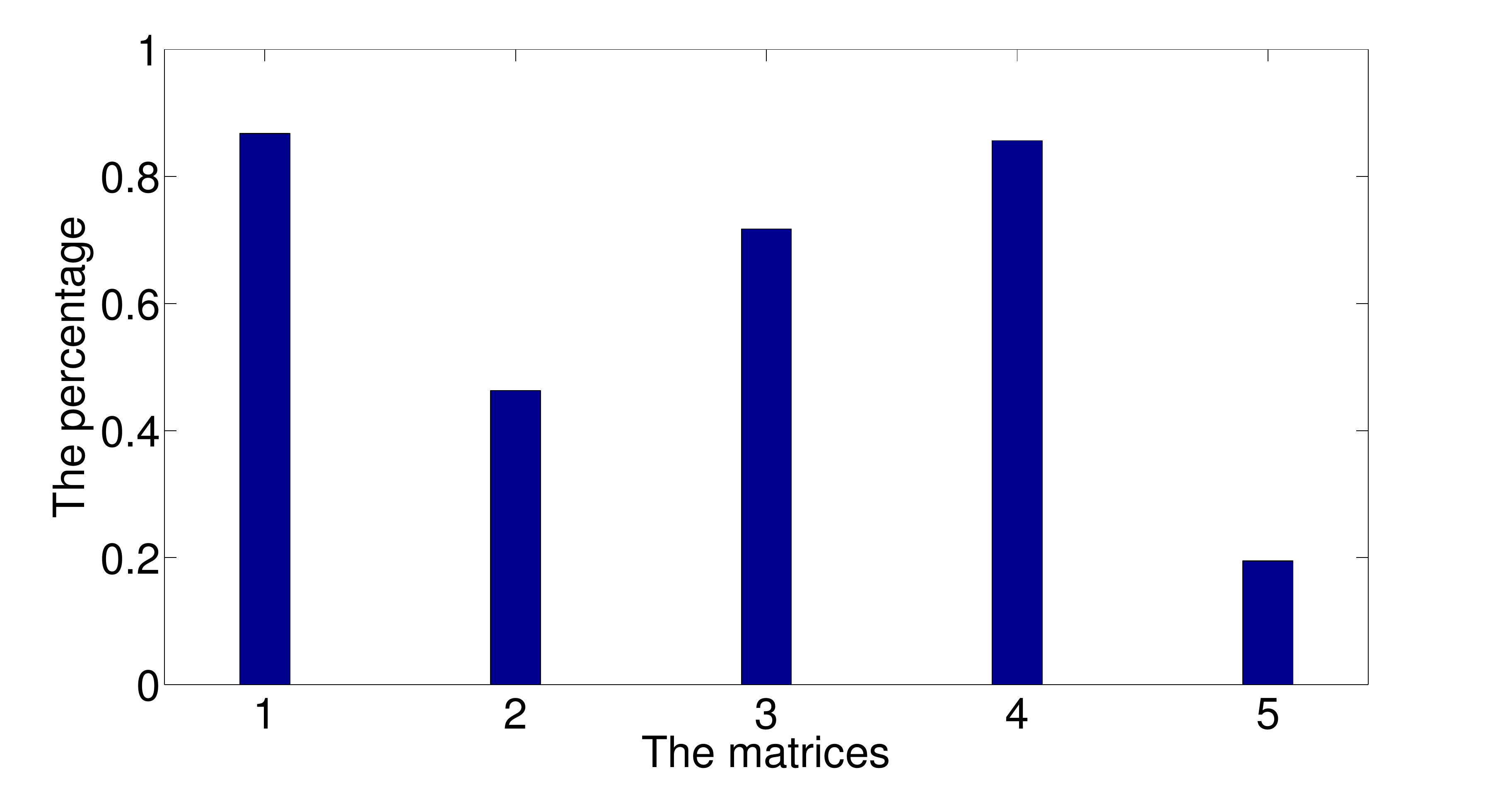}
 \label{fig:percent}}
 \caption{Results for these difficult matrices}
 \label{fig:tough}%
 \end{figure}

 \begin{figure}[ptbh]
 \centering
 \subfigure[The locations of $d_{\min}$ when finding the first 500 singular values]{
 \includegraphics[totalheight=2.0in,width=2.25in]{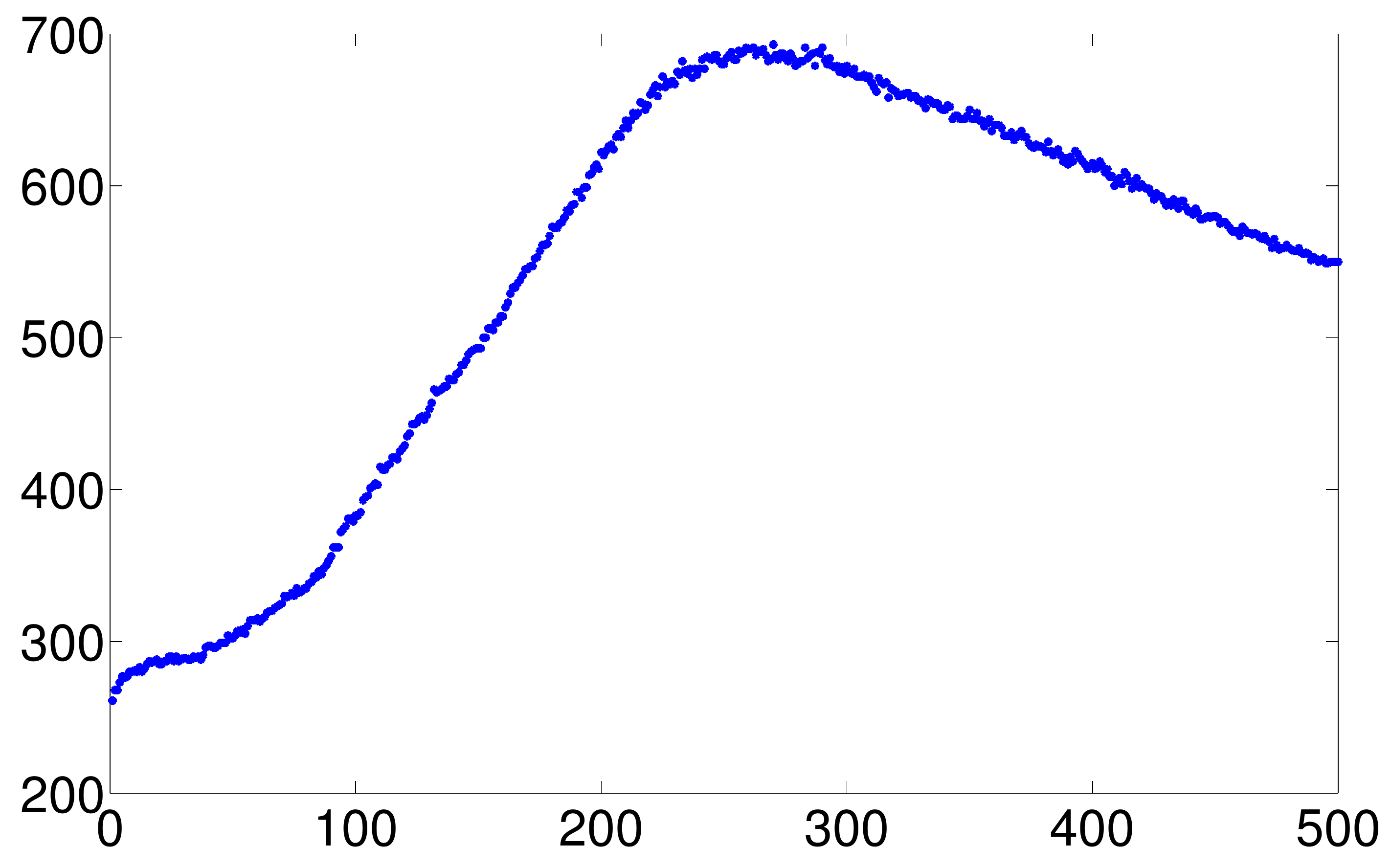}
 \label{fig:first500}}
 \subfigure[The number of iterations for the first 10 singular values]{
 \includegraphics[height=2.02in,width=2.40in]{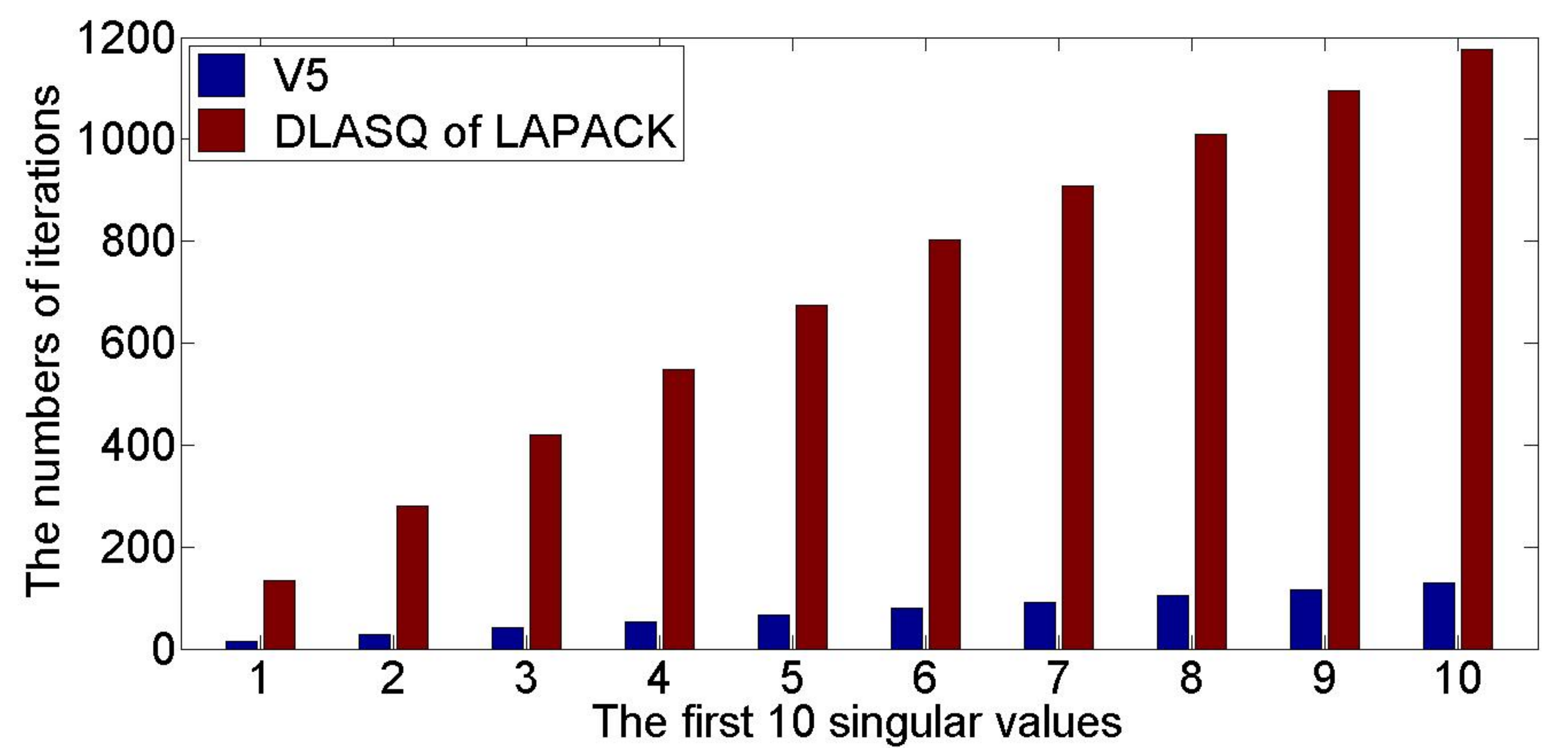}
 \label{fig:first10it}}
 \caption{More results for Matrix\_3 in Table~\ref{tab:matrices}}
 \label{fig:FirstIt}%
 \end{figure}

\subsection{Some more tests}

To further show the improvement of V5, we use some more
matrices to test our algorithm.  These matrices can be
divided into two classes: 1) matrices from applications; 2)
matrices constructed for testing the dqds algorithm.

These matrices are illustrated in Table~\ref{tab:matrices2}.  The last
four are from industrial applications which are collected in LAPACK
tester {\tt stetester} and can be obtained from a website maintained
by O.~Marques~\cite{MarquesWebsite}. To measure the computing time
accurately, we intentionally choose some big matrices.

The results are shown in Figure~\ref{fig:bar2} and
Figure~\ref{fig:bar3}.  Figure~\ref{fig:bar2} shows the ratios over
DLASQ in terms of time.  For these matrices from industrial
applications V5 saves about 30\% in time.  For these glued matrices
which are difficult for the dqds algorithm, V5 has even more speedups,
about 2x--5x times faster.  Figure~\ref{fig:bar3} provides more
information. For these matrices, d-deflation strategy deflates about
40\% of all singular values.  Figure~\ref{fig:bar3} further supports
our conclusion that the d-deflation strategy greatly improves the
performance of the dqds algorithm.

\begin{table}[ptb]
\caption{More bidiagonal matrices}%
\label{tab:matrices2}
\begin{center}%
\begin{tabular}
[c]{|cllc|}\hline
& $n$  & Description of the bidiagonal matrix $B$ & Source \\ \hline \hline
% 1 & 8000  & $\sqrt{q_i}=n+1-i,\sqrt{e_i}=1$ &  \cite{dqds-Agg} \\
% 2 & 8000 & Cholesky factor of tridiagonal (1,2,1) matrix & \cite{A880} \\
% 3 & 8000 & Cholesky factor of Wilkinson-type matrix & \cite{A880} \\
% 4 & 8000 & Cholesky factor of Clement-type matrix & \cite{A880} \\ \hline \hline
1 & 15,005 & Cholesky factor of Glued Wilkinson matrix & \cite{A880} \\
2 & 15,005 & Cholesky factor of Glued Clement matrix & \cite{A880} \\
3 & 30,010 & Cholesky factor of Glued Wilkinson matrix & \cite{A880} \\
4 & 30,010 & Cholesky factor of Glued Clement matrix & \cite{A880} \\ \hline \hline
5  & 2901  & Cholesky factor of T\_nasa2910\_1.dat & \cite{MarquesWebsite} \\
6 & 3258  & Cholesky factor of T\_bcsstkm10\_3.dat & \cite{MarquesWebsite} \\
7 & 4098  & Cholesky factor of T\_sts4098\_1.dat  & \cite{MarquesWebsite} \\
8 & 5472  & Cholesky factor of T\_nasa1824\_3.dat & \cite{MarquesWebsite} \\\hline
\end{tabular}
\end{center}
\end{table}

For these 254 matrices in {\tt stetester} for which DLASQ needs more
than $1.5e$-$2$ second, V5 usually saved about 20\% in time, and was
slower than DLASQ only for 32 of them but never slower than DLASQ by
more than $0.0145$ second.  Among these 254 matrices, V5 was 3x faster
for 4 of them; 2x faster for 14 of them; 1.5x faster for 47 of them;
1.2x faster for 174 of them.  The results are shown in
Figure~\ref{fig:M5}.

\begin{figure}[ptbh]
\centering
\subfigure[The time ratio of V5 over DLASQ ]{
\includegraphics[height=2.20in, width=2.35in]{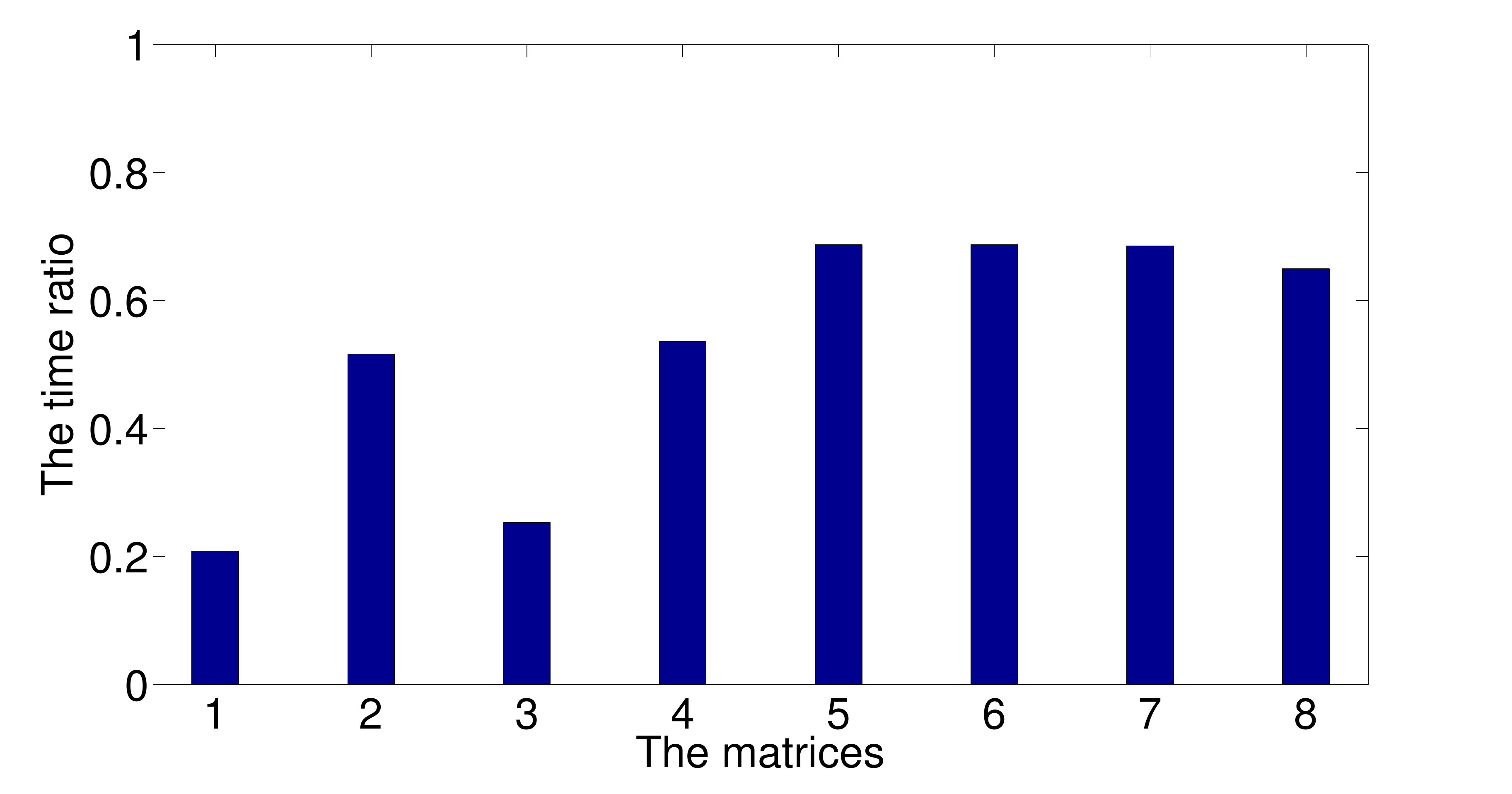}
\label{fig:bar2}}
\subfigure[The percentage of singular values deflated by d-deflation strategy in section~\ref{sec:dmin}]{
\includegraphics[height=2.20in,width=2.35in]{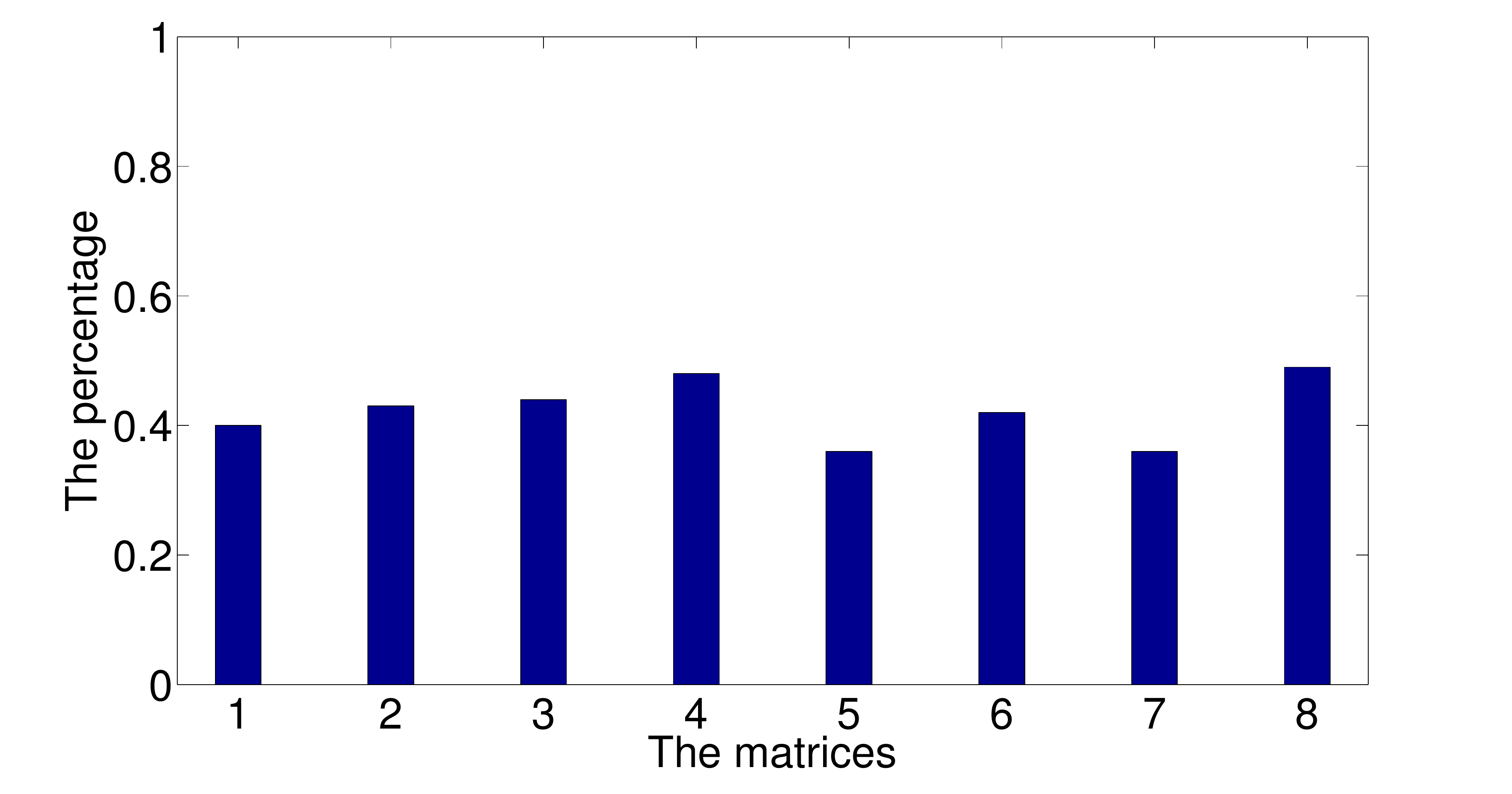}
\label{fig:bar3}}
\caption{More results for matrices from construction and industry}%
\label{fig:more}%
\end{figure}

\begin{figure}[ptbh]
\centering
\includegraphics[height=2.0in, width=4.0in]{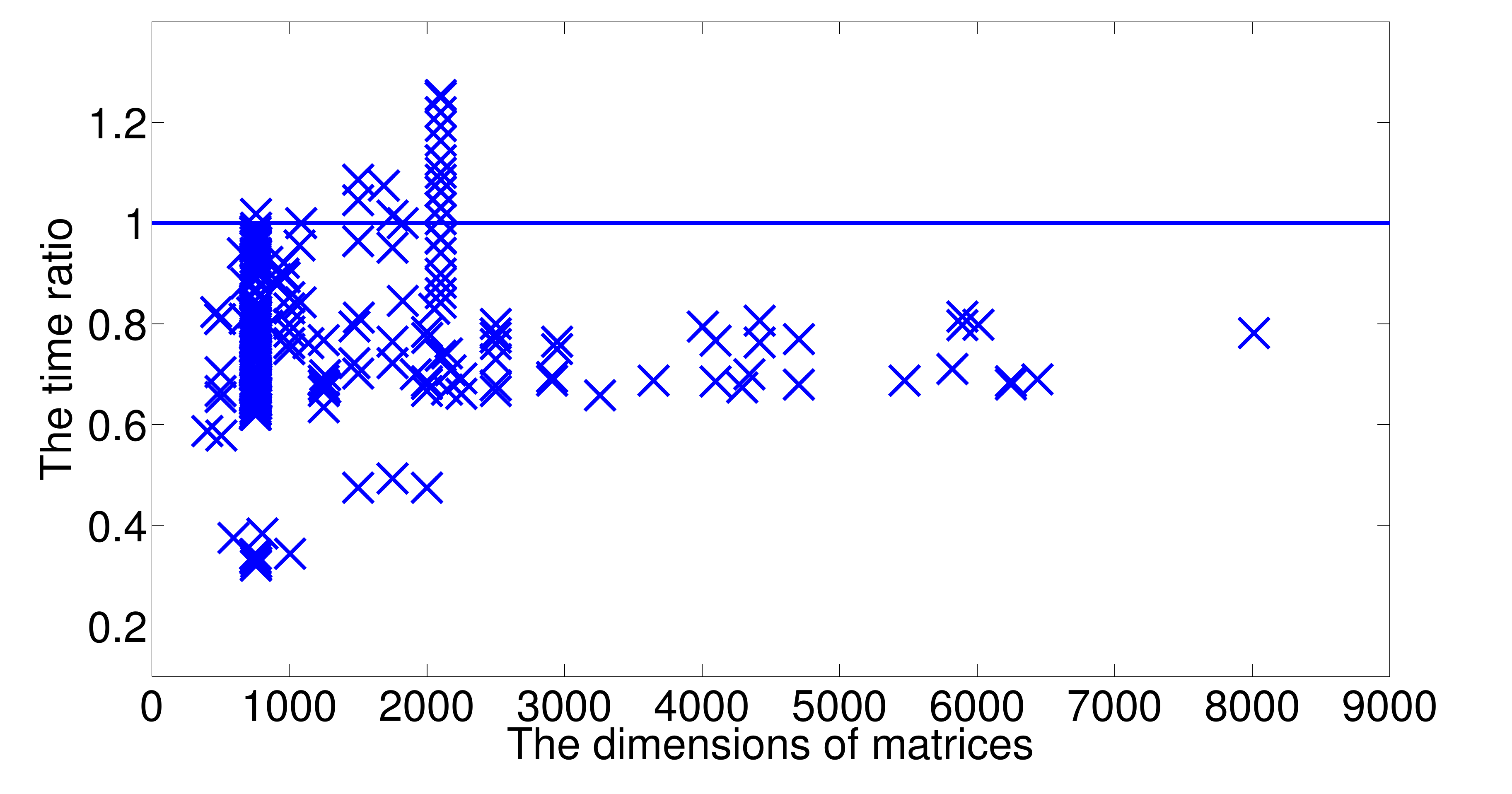}
\caption{The time of V5 vs DLASQ for matrices in LAPACK tester}%
\label{fig:M5}%
\end{figure}

\subsection{The test of accuracy}

To show the accuracy of the singular values computed by our algorithm,
we compared the singular values obtained by our algorithm with those
gotten from a bisection algorithm.  We first change the singular value
problem of a bidiagonal matrix into the eigenvalue problem of a
tridiagonal matrix of double size with zero diagonals, and then use a
bisection algorithm to find its eigenvalues.  We assume the
eigenvalues computed by the bisection algorithm are `correct'.  The
results of maximum relative error are shown in
Figure~\ref{fig:maxrelerr}.  Figure~\ref{fig:norm2err} shows the
2-norm of the relative errors.

\begin{figure}[ptbh]
\centering
\subfigure[Maximum relative error]{
\includegraphics[height=2.20in, width=2.45in]{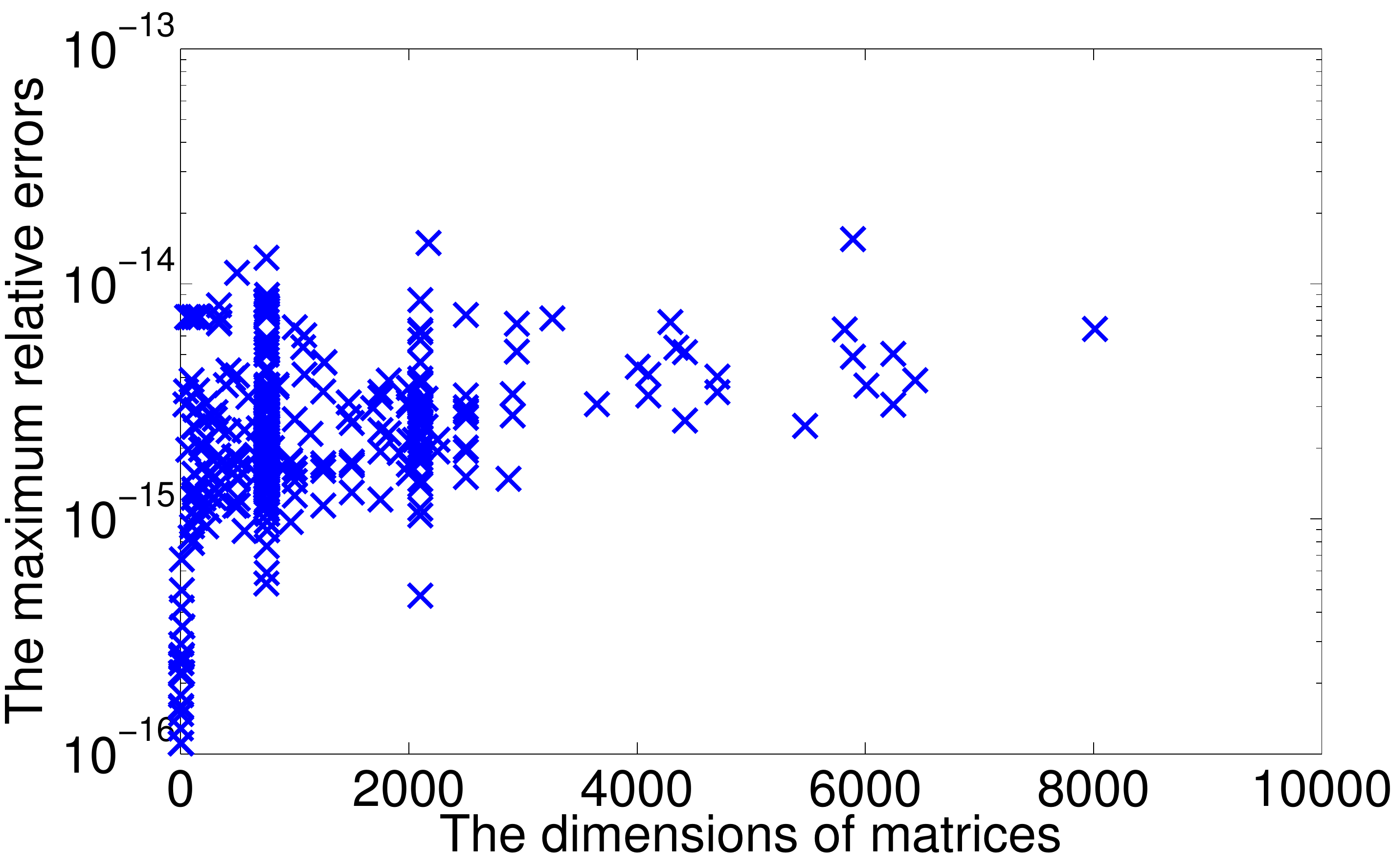}
\label{fig:maxrelerr}}
\subfigure[2-norm error]{
\includegraphics[height=2.20in,width=2.45in]{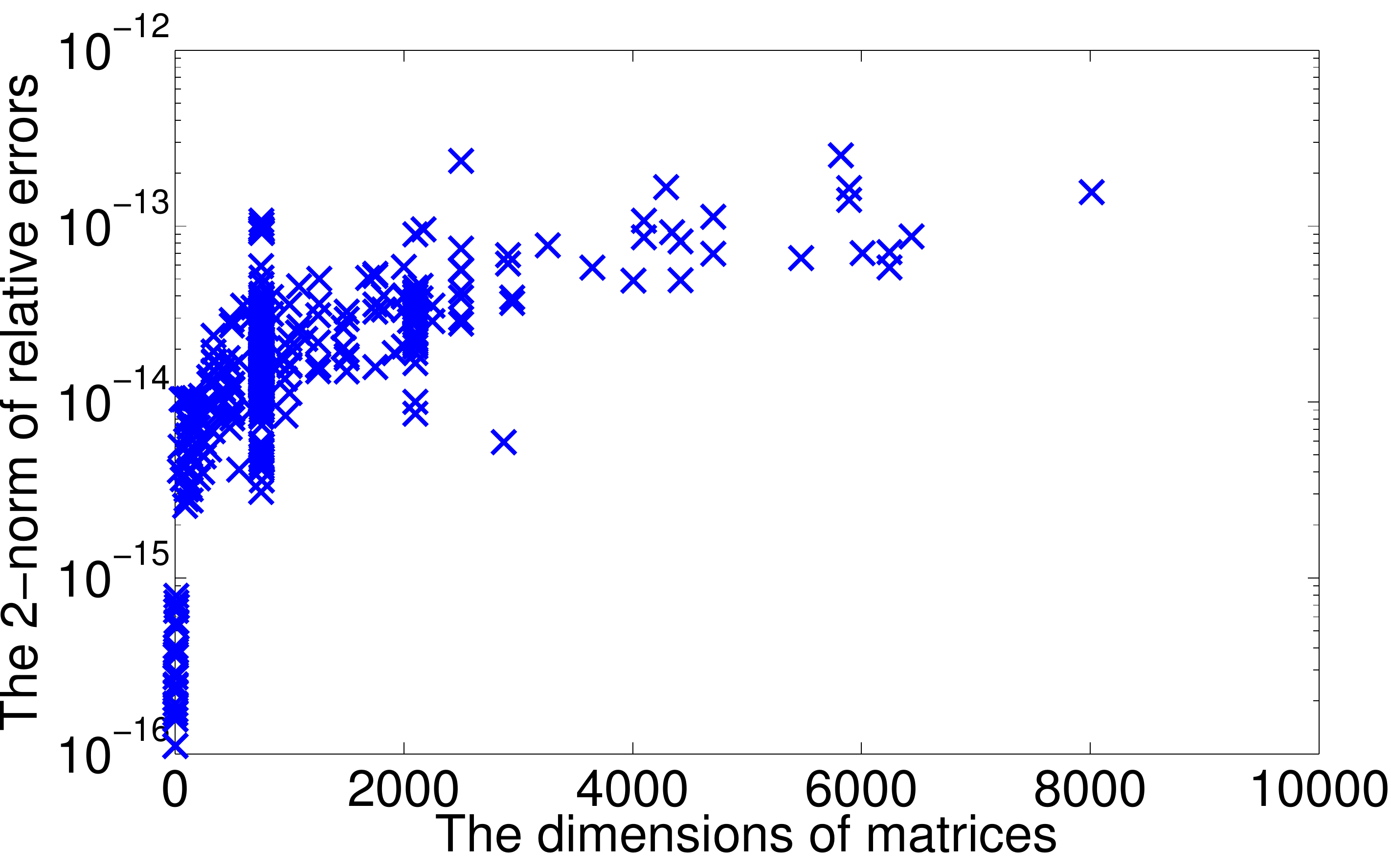}
\label{fig:norm2err}}
\caption{The comparison of accuracy with bisection}%
\label{fig:accuracy}%
\end{figure}

% We also compare the accuracy with DLASQ.
% The results are that our new algorithm is usually more accurate
% than DLASQ since it usually needs fewer iterations.

% The results
% from Figure~\ref{fig:maxerrla} and
% Figure~\ref{fig:norm2errla} show our algorithm is usually
% more accurate than DLASQ.

% \begin{figure}[ptbh]
% \centering
% \subfigure[Maximum relative error]{
% \includegraphics[height=2.20in, width=2.45in]{../jpg/Max-rel-my-la}
% \label{fig:maxerrla}}
% \subfigure[2-norm relative error]{
% \includegraphics[height=2.20in,width=2.45in]{../jpg/Norm2-rel-my-la}
% \label{fig:norm2errla}}
% \caption{The comparison of accuracy with DLASQ}%
% \label{fig:accuracy2}%
% \end{figure}

\subsection{Comparisons with aggressive early deflation}

The aggressive early deflation strategy~\cite{dqds-Agg} is designed to
enhance the dqds algorithm.  This deflation strategy can help a lot
for matrices which are easy for dqds.  In this subsection, we
demonstrate that AED does not help for those difficult matrices in
Table~\ref{tab:matrices}.  There are two versions of AED
in~\cite{dqds-Agg}, denoted by AggDef1 and AggDef2,
see~\cite{dqds-Agg} for details.  Since AggDef2 is more efficient than
AggDef1~\cite{dqds-Agg}, we only compare our algorithms with AggDef2.

When comparing the accuracy, we also assume the results by the
bisection method are correct.  From the results in
Table~\ref{tab:Aggdef}, we can see that for these difficult matrices
AggDef2 has nearly the same performance as DLASQ.  Our improved
algorithm (V5) is the most accurate and fastest among these three
versions.

%For this random matrix with dimension 5000, AggDef1 and AggDef2 nearly does not help either.
\begin{table}[ptb]%
\caption{Comparison of V5 with AggDef2~\cite{dqds-Agg} for matrices in Table~\ref{tab:matrices}}
\label{tab:Aggdef}
\begin{center}%
\begin{tabular}{|c|c|ccccc|}
\hline
 &Methods&  Matrix\_1 &    Matrix\_2  &   Matrix\_3     &   Matrix\_4    &   Matrix\_5  \\ \hline \hline
$n$ &  & 544 & 1000  & 1087 & 1088 & 5000  \\ \hline \hline
Speedup    &   AggDef2  & 1.06  & 1.05 & 1.01   & 1.08   & 1.02   \\
$\frac{Time(DLASQ)}{Time(method)}$
       &   V5     & 2.83   & 2.74 & 9.99  & 3.31  & 1.20   \\ \hline \hline
\multirow{2}{*}{Max. Rel.}
       &    DLASQ    & 6.22e-15 &  9.54e-15 & 4.07e-14 & 2.42e-14  & 9.47e-15  \\
       &   AggDef2  & 6.66e-15 &  9.33e-15 & 3.97e-14 & 2.35e-14  & 9.52e-15  \\
Error  &   V5     & 3.66e-15 &  7.99e-15 & 3.85e-15 & 5.66e-15  & 6.27e-15  \\ \hline
\end{tabular}
\end{center}
\end{table}

% %To compare the timing, we treat DLASQ as the standard method.
% When comparing the accuracy, we also assume the results by the bisection method are correct.
% From the results in Table~\ref{tab:Aggdef}, we can see that for these difficult matrices
% AggDef2 has nearly the same performance as DLASQ.  Our improved algorithm (V5) is the most
% accurate and fastest among these three versions.
% %For this random matrix with dimension 5000, AggDef1 and AggDef2 nearly does not help either.
% \begin{table}[ptb]%
% \caption{Comparison of V5 with AggDef2~\cite{dqds-Agg} for matrices in Table~\ref{tab:matrices}}
% \label{tab:Aggdef}
% \begin{center}%
% \begin{tabular}{|c|c|ccccc|}
% \hline
%  &Methods&  Matrix\_1 &    Matrix\_2  &   Matrix\_3     &   Matrix\_4    &   Matrix\_5  \\ \hline \hline
% $n$ &  & 544 & 1000  & 1087 & 1088 & 5000  \\ \hline \hline
% Speedup    &   AggDef2  & 1.06  & 1.05 & 1.01   & 1.08   & 1.02   \\
% $\frac{Time(DLASQ)}{Time(method)}$
%        &   V5     & 2.83   & 2.74 & 9.99  & 3.31  & 1.20   \\ \hline \hline
% \multirow{2}{*}{Max. Rel.}
%        &    DLASQ    & 6.22e-15 &  9.54e-15 & 4.07e-14 & 2.42e-14  & 9.47e-15  \\
%        &   AggDef2  & 6.66e-15 &  9.33e-15 & 3.97e-14 & 2.35e-14  & 9.52e-15  \\
% Error  &   V5     & 3.66e-15 &  7.99e-15 & 3.85e-15 & 5.66e-15  & 6.27e-15  \\ \hline
% \end{tabular}
% \end{center}
% \end{table}

By combining our improvements with AggDef2, we obtain a hybrid
algorithm, HDLASQ, which is summarized in Algorithm~\ref{alg:hdlasq},
similar to the algorithm in~\cite{dqds-Agg}.

\fbox{
\begin{minipage}[t]{12cm}
\begin{algorithm}[HDLASQ]
\label{alg:hdlasq}  \\
{\bf Inputs:} bidiagonal matrix $B\in R^{n\times n}$, deflation frequency $p$
\begin{enumerate}
\item {\bf while} size of current segment $B$ is larger than $\sqrt{n}$ {\bf do}
\item \hspace{0.3cm} run $p$ iterations of dqds by calling V5;
\item \hspace{0.3cm} perform aggressive early deflation;
\item {\bf end while}
\item run dqds until all singular values are computed.
\end{enumerate}
\end{algorithm}
\end{minipage} }

As mentioned in~\cite{dqds-Agg}, setting $p$ too large may
deteriorate the rate of convergence, but the performance depends not much
on $p$.  For example, HDLASQ with $p=200$ was about 10\% slower than
letting $p=20$ for matrices in Table~\ref{tab:easy} (except Mat1 for
which AggDef2 is particularly effective).  In our
implementation\footnote{We used the codes of AggDef,
available at: http://www.opt.mist.i.u-tokyo.ac.jp/\~{}nakatsukasa}, we let $p=50$.

For these disordered matrices in Table~\ref{tab:matrices}, HDLASQ has
nearly the same performance as V5.  For the matrices that are good
for AED (shown in Table~\ref{tab:easy}), the comparison results of
HDLASQ with other methods are shown in Table~\ref{tab:hdlasq}, from
which we can see that HDLASQ is faster than AggDef2, DLASQ and V5.
Furthermore, Figure~\ref{fig:hdlasq} shows the time ratio of HDLASQ
over DLASQ for matrices in {\tt stetester} for which DLASQ needs more
than $1.5e$-$2$ second.  By comparing Figure~\ref{fig:hdlasq} with
Figure~\ref{fig:M5}, we can also see that HDLASQ is usually better
than V5.

\begin{table}[ptb]
\caption{Some bidiagonal matrices for testing HDLASQ}%
\label{tab:easy}
\begin{center}%
\begin{tabular}
[c]{|cllc|}\hline
Matrix& $n$  & Description of the bidiagonal matrix $B$ & Source \\ \hline \hline
 Mat1 & 30000  & $\sqrt{q_i}=n+1-i,\sqrt{e_i}=1$ &  \cite{dqds-Agg} \\
 Mat2 & 30000  & $\sqrt{q_i}=n+1-i,\sqrt{e_i}=\sqrt{q_i}/5$ &  \cite{dqds-Agg} \\
 Mat3 & 30000 & Toeplitz: $\sqrt{q_i}=1$, $\sqrt{e_i}=2$ & \cite{A880,dqds-Agg} \\
 Mat4 & 30000 & Cholesky factor of tridiagonal (1,2,1) matrix & \cite{A880,Fernando94} \\  \hline
\end{tabular}
\end{center}
\end{table}

\begin{table}[ptb]%
\caption{The comparisons of HDLASQ with AggDef2 and V5 for matrices in Table~\ref{tab:easy}}
\label{tab:hdlasq}
\begin{center}
\begin{tabular}{|c|c|cccc|}
\hline
& Methods &  Mat1 &  Mat2  &   Mat3  &  Mat4    \\ \hline \hline
$n$ & & 30000 &  30000 &  30000  & 30000  \\ \hline \hline
Speedup  &   V5   & 1.66   & 1.48  & 1.39  & 1.40  \\
\multirow{2}{*}{$\frac{Time(DLASQ)}{Time(method)}$}
&   AggDef2  & 76.1  & 1.99   & 1.57   & 1.59   \\
&   HDLASQ  & 79.5  & 2.18   & 1.67   & 1.69    \\ \hline
\end{tabular}
\end{center}
\end{table}

\begin{figure}[ptbh]
\centering
\includegraphics[height=2.10in, width=4.0in]{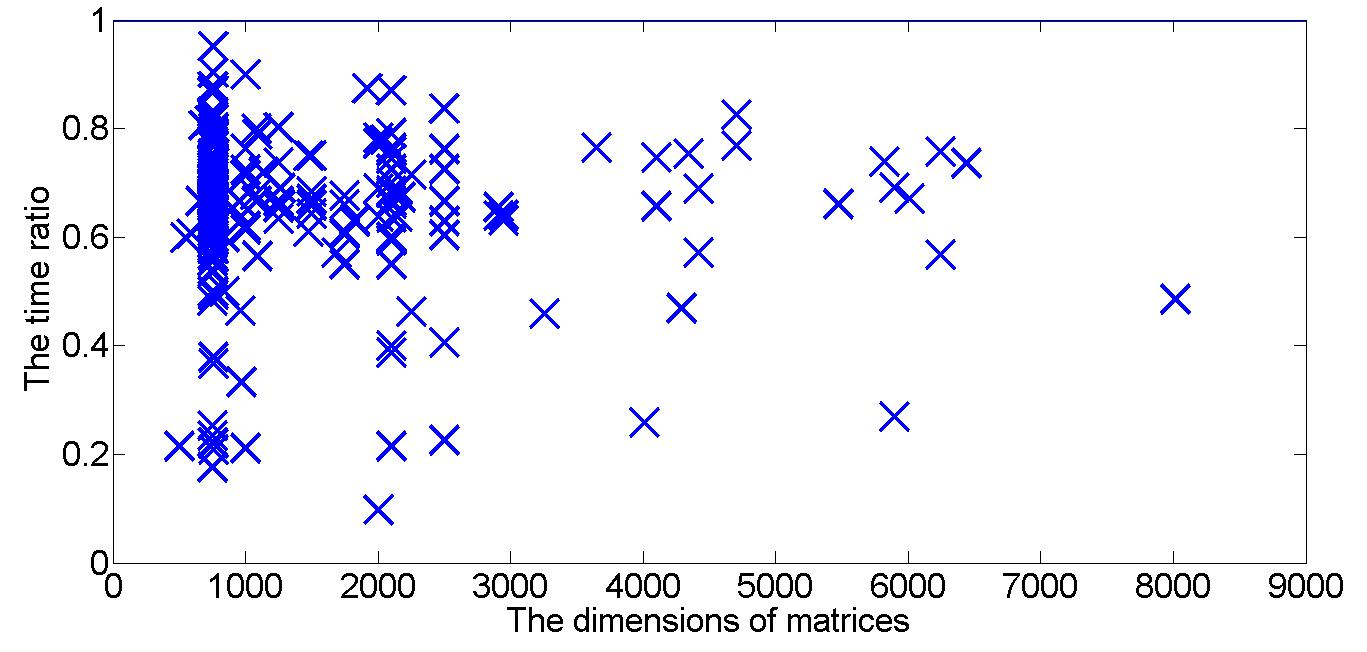}
\caption{The time of HDLASQ vs DLASQ for matrices in LAPACK tester}
\label{fig:hdlasq}%
\end{figure}

\section{Conclusions}
\label{sec:conclusion}
In this paper we first propose a novel deflation strategy for the dqds
algorithm, named \emph{d-deflation strategy}, which is different from
the classic deflation strategies and greatly improves the dqds
algorithm.  Note that the shifts for the dqds algorithm must be a
fraction of the upper bound $sup$.  Together with the technique of
\emph{updating the upper bound}, they ensure the linear worst case
complexity of our improved dqds algorithm V5.  Additional modifications
to certain shift strategies are also included. These improvements
together make V5 up to 10x faster for those difficult matrices and
1.2x-4x faster in general without any loss of accuracy. By combining
V5 with AED, we propose a hybrid algorithm (HDLASQ) which is shown to
be faster than DLASQ, V5 and AggDef2.

\section*{Acknowledgement}
The authors are very grateful to the anonymous referees and
the associated editor Chen Greif for their valuable suggestions, and would like to
acknowledge many helpful discussions with J. Demmel,
W. Kahan, and B. Lipshitz. The authors also thank O. Marques for
making his collection of difficult bidiagonal matrices
available to us for numerical experiments.

%\bibliographystyle{plain}
%\bibliography{../Fifth/thesis-refer}

\appendix
\section*{}
In this Appendix, we prove Theorem~\ref{thm:mainconv}.  Our techniques
are similar to those used in~\cite{Aishima-Convergence} to prove the
order of convergence of the Johnson shift.
We assume that matrix $B$ is defined as in~\eqref{eq:B} and that its entries
satisfy $a_k=\sqrt{q_k}=\sqrt{q_k^{(0)}}$, $b_k=\sqrt{e_k}=\sqrt{e_k^{(0)}}$.
In the following context, let $\{q_k^{(l)},e_k^{(l)}\}$ denote the array
after $l$ dqds transforms.

We first introduce the reverse stationary differential qd algorithm with shifts
(\emph{dstqds}), which starts from the bottom $q_n$, illustrated as
follows~\cite{dqds-Agg}.
Let $\{\overset{o}{q},\overset{o}{e}\}$ denote the array after
doing one dstqds transform on the $\{q,e\}$ array.

\fbox{
\begin{minipage}[t]{12cm}
\begin{algorithm}[reverse dstqds]
\label{alg:dstqds2}
\begin{description}
\item \hspace{0.5cm} $t_n = -s$
\item \hspace{0.5cm} {\bf for} $i=n-1,n-2,\ldots,1$
\item \hspace{1.0cm} $\overset{o}{q}_{i+1} = q_{i+1} +t_{i+1} $
\item \hspace{1.0cm} $tmp= \frac{e_{i}}{\overset{o}{q}_{i+1}} $
\item \hspace{1.0cm} $ \overset{o}{e}_i = q_{i+1}* tmp$
\item \hspace{1.0cm} $ t_{i} = t_{i+1} *tmp - s $
\item \hspace{0.5cm} {\bf end for}
\item \hspace{0.5cm} $\overset{o}{q}_1 = q_1 +t_1$
\end{description}
\end{algorithm}
\end{minipage} }

For the $\{q, e\}$ array, if we do the dqds transformation starting
from $q_1$ and dstqds starting from $q_n$, there is enough
information to form the twisted factorization at any $k$ we
choose (see~\cite{Dhillon-RelGap} for details):
\begin{equation}
\label{eq:twistN}
BB^T-sI=\mathcal{N}\mathcal{N}^T,
\end{equation}
where
\[
\mathcal{N}^T = \begin{bmatrix}
\sqrt{q_1^{(1)}} & \sqrt{e_1^{(1)}} &&&&&& \\
& \cdotp & \cdotp &&&& \\
&0 & \sqrt{q_{k-1}^{(1)}} & \sqrt{e_{k-1}^{(1)}} &&&& \\
&&0 & \sqrt{\gamma_k}  &&&& \\
&&& \sqrt{\overset{o}{e}_{k}} & \sqrt{\overset{o}{q}_{k+1}} & 0 &&\\
% &&&& \sqrt{\overset{o}{e}_{k+1}} & \sqrt{\overset{o}{q}_{k+2}} & & \\
&&&& \cdotp & \cdotp &  \\
&&&&& \sqrt{\overset{o}{e}_{n-1}} & \sqrt{\overset{o}{q}_{n}}
 \end{bmatrix},
\]
 and $\gamma_k=
d_k+t_{k+1}\frac{e_k}{\overset{o}{q}_{k+1}}$.

Recall that ${\bf e}_k$ denotes the $k$-th column of an $n\times n$
identity matrix (to distinguish from $e_k$).  The solution to
$\mathcal{N}\mathcal{N}^T z^{(1)} = \gamma_k {\bf e}_k$ can be computed by
\begin{equation}
\label{eq:Z}
z^{(1)}(j) = \left\{
\begin{array}{ll} 1 & \mbox{if} \quad j = k , \cr
-z^{(1)}(j+1)\sqrt{e_j^{(1)}/q_j^{(1)}}, &
  \mbox{if} \quad j < k, \\
 -z^{(1)}(j-1)
  \sqrt{\overset{o}{e}_{j-1}/\overset{o}{q}_{j}}, &
  \mbox{if} \quad j > k.\end{array}\right.
\end{equation}
Denote $(\varphi^{(1)})^2 =\|z^{(1)}\|^2-1$.  By Theorem 4.5.1~\cite{Parlett-book}
(or equation (12) in~\cite{Parlett-Imp}), the smallest eigenvalue of
$BB^T-sI$ is bounded below by
\begin{equation}
\label{eq:bound1}
\phi= \gamma_k \frac{1-\varphi^{(1)}}{1+(\varphi^{(1)})^2},
\end{equation}
which is used as the shift for the next dqds transform in DLASQ if
$\varphi^{(1)}$ is small (for example, less than $\frac{3}{4}$).  The
equation~\eqref{eq:Z} only needs to be solved approximately, see
section 6.3.3 of~\cite{Parlett-Imp} for details.

Below we show that for sufficiently large $l$ the shifts are chosen by
the twisted shift strategy, and then prove its order of
convergence. We introduce a result from~\cite{Aishima-Convergence},
which states the convergence of the dqds algorithm.
\begin{lemma}[Convergence of the dqds algorithm \cite{Aishima-Convergence}]
\label{thm:converg}
Suppose the matrix $B$ defined as in \eqref{eq:B} has positive nonzero elements,
and the shift in the dqds algorithm is taken so that $0 \leq s^{(l)} < (\sigma_{\min}^{(l)})^2$ holds. Then
\begin{equation}
\label{eq:sn}
\sum_{l=0}^{\infty} s^{(l)} \leq \sigma_n^2.
\end{equation}
Moreover,
\begin{eqnarray}
& \lim_{l \rightarrow \infty}& e_k^{(l)}  =  0 \quad (k=1,2,\cdots,n-1), \label{eq:ek}\\
&\lim_{l \rightarrow \infty}& \frac{e_k^{(l+1)}}{e_k^{(l)}}  =  \rho_k, \quad (k=1,\cdots,n-1),
\label{eq:snd}\\
&\lim_{l \rightarrow \infty}& q_k^{(l)} = \sigma_k^2 -\sum_{l=0}^{\infty} s^{(l)} \quad (k=1,2,\cdots,n), \label{eq:qk}
\end{eqnarray}
where $\rho_k = \frac{\sigma_{k+1}^2 - \sum_{l=0}^{\infty}s^{(l)}}{\sigma_{k}^2 - \sum_{l=0}^{\infty}s^{(l)}}$.
\end{lemma}

\begin{remark} Since $\sigma_{k+1}^2<\sigma_k^2$
and $\rho_k = \frac{\sigma_{k+1}^2
  -\sum_{l=0}^{\infty}s^{(l)}}{\sigma_{k}^2 -
  \sum_{l=0}^{\infty}s^{(l)}}$, we have $\rho_k < 1$ and
therefore ${e_k^{(l+1)}} < {e_k^{(l)}}$ for all large enough
$l$.
\end{remark}

Based on Lemma \ref{thm:converg}, we can prove the following lemma.
\begin{lemma}
\label{lem:prov1}
Under the same assumptions as in Lemma \ref{thm:converg},
for all sufficiently large $l$,
the shifts will be chosen by the twisted shift strategy,
\begin{equation}
\label{eq:prov1}
s^{(l)} = \gamma_n^{(l)} \frac{1-\varphi^{(l)}}{1+(\varphi^{(l)})^2},
\end{equation}
where $\gamma_n^{(l)} = d_{\min}^{(l)} = q_n^{(l)}$ and $\varphi^{(l)}$ can
be computed by~\eqref{eq:Z}.
\end{lemma}

\begin{proof}
By~\cite{Fernando94} the smallest singular value of $B$ in
\eqref{eq:B} satisfies $\sigma_n > 0$.
%are all positive and different from each other~\cite{Fernando94}.
By equations~\eqref{eq:ek} and~\eqref{eq:qk}, we know that $e_k^{(l)} \rightarrow 0$, $q_k^{(l)}\geq \sigma_k^2-\sigma_n^2 >0$
(as $l \rightarrow \infty$ for $k=1,2,\cdots,n-1$)
and that $\{q_k^{(l)}\}$ gradually become monotonic with respect to $k$.
For sufficiently large $l$ we will have dmink$=n$.
By \eqref{eq:Z}, we get
\begin{equation}
\label{eq:Z2}
\begin{split}
z^{(l)}(n) & = 1, \\
z^{(l)}(j) & = -z^{(l)}(j+1)\sqrt{{e}_j^{(l)}/{q}_j^{(l)}}, j < n.
\end{split}
\end{equation}
Let $(z^{(l)})^T = ((x^{(l)})^T \quad 1 )$ and then, after
some algebraic manipulations,
\begin{equation}
  \label{eq:x}
%  (\varphi^{(l)})^2=\|x^{(l)}\|^2=\frac{e_{n-1}^{(l+1)}}{q_{n-1}^{(l+1)}}(1+\frac{e_{n-2}^{(l+1)}}{q_{n-2}^{(l+1)}}(1+\frac{e_{n-3}^{(l+1)}}{q_{n-3}^{(l+1)}}(1+\cdots))).
(\varphi^{(l)})^2=\|x^{(l)}\|^2=\frac{e_{n-1}^{(l)}}{q_{n-1}^{(l)}}(1+\frac{e_{n-2}^{(l)}}{q_{n-2}^{(l)}}(1+\frac{e_{n-3}^{(l)}}{q_{n-3}^{(l)}}(1+\cdots))).
\end{equation}
% We only need to prove that $\varphi^{(l)}$ will be very small for sufficiently large $l$.
By equations~\eqref{eq:ek} and~\eqref{eq:qk},
we know that $e_{k}^{(l)}\rightarrow 0$, $q_k^{(l)}$ converges to a positive
constant as $l \rightarrow \infty$
(for $1\le k \le n-1$) and that
$\varphi^{(l)}$ also converges to zero.
%a sufficiently large $l \equiv N$.
%The shift for the $N$th dqds transform will be chosen by~\eqref{eq:prov1}.
%
Therefore, the shifts would be chosen by~\eqref{eq:prov1} for all sufficiently large
$l$.
\end{proof}

\begin{remark}
If $\varphi^{(l)}\le \frac{3}{4}$, DLASQ will use~\eqref{eq:prov1} as a shift.
To prove the main conclusion of this appendix, we assume that
$d_{\min}$ always moves to the bottom.  The aim is analyzing
the asymptotic convergence rate of the twisted shift strategy.
The following lemma reveals another good property of the twisted
shift.  It is said that $d_{\min}=q_n^{(l)}$ always converges
to zero if using the twisted shift.
\end{remark}
\begin{lemma}
\label{lem:prov2}
Under the same assumptions as in Lemma \ref{thm:converg}, we have
\[
\sum_{l=0}^{\infty} s^{(l)} = \sigma_n^2,
\]
\[
\lim_{l\rightarrow \infty} q_k^{(l)} = \sigma_k^2 - \sigma_n^2 \quad  (k=1,\cdots,n-1);
\lim_{l\rightarrow \infty} q_n^{(l)} = 0.
\]
\end{lemma}

\begin{proof} From Lemma \ref{lem:prov1}, we know $\|x^{(l)}\|
\rightarrow 0$ and $$ \lim_{l\rightarrow \infty} s^{(l)} =
\lim_{l\rightarrow \infty} q_n^{(l)} \geq 0. $$

Furthermore, since $\lim_{l\rightarrow \infty} s^{(l)} = 0$ from
equation \eqref{eq:sn}, we have
$\lim_{l\rightarrow \infty} q_n^{(l)}=0.$ The first two equations
follow from $\lim_{l\rightarrow \infty} q_n^{(l)}=0$ and
equation~\eqref{eq:qk}.
\end{proof}

We are now ready to prove {\bf Theorem~\ref{thm:mainconv}}.
\begin{proof}
% We compute the rate of convergence of $e_{n-1}^{(l)}$.
By Algorithm \ref{alg:dqds}, we have
\[
\begin{split}
q_n^{(l+1)} &= d_{n-1}^{(l+1)}\frac{q_n^{(l)}}{q_{n-1}^{(l+1)}}-s^{(l)}
    = q_n^{(l)}-e_{n-1}^{(l)}\frac{q_n^{(l)}}{q_{n-1}^{(l+1)}}-s^{(l)} \\
    &= q_n^{(l)}-e_{n-1}^{(l+1)}-s^{(l)}.
\end{split}
\]
By Lemma \ref{lem:prov1} the shift is chosen by the twisted shift strategy for
sufficiently large $l$, and we have
\[
\begin{split}
q_n^{(l+1)} &= q_n^{(l)}-e_{n-1}^{(l+1)}-q_n^{(l)}\frac{1-\varphi^{(l)}}{1+(\varphi^{(l)})^2} \\
    & = q_n^{(l)}\frac{(\varphi^{(l)})^2+\varphi^{(l)}}{1+(\varphi^{(l)})^2}-e_{n-1}^{(l+1)},
\end{split}
\]
where
$
\varphi^{(l)}=\|x^{(l)}\| %=\sqrt{\frac{e_{n-1}^{(l)}}{q_{n-1}^{(l)}}}\cdotp \sqrt{1+\frac{e_{n-2}^{(l)}}{q_{n-2}^{(l)}} \Delta_1}
    = \sqrt{\frac{e_{n-1}^{(l)}}{q_{n-1}^{(l)}}}\cdotp \sqrt{1+\xi^{(l)}},
$
and $\xi^{(l)}\to 0$ as $l \rightarrow \infty$ (by equation~\eqref{eq:x}).

Hence
\begin{equation}\label{qn}
q_n^{(l+1)} = q_n^{(l)} \sqrt{e_{n-1}^{(l)}}
\frac{\left(1+\|x^{(l)}\|\right)\sqrt{1+\xi^{(l)}}}{\left(1+\|x^{(l)}\|^2\right)
\sqrt{q_{n-1}^{(l)}}}-e_{n-1}^{(l+1)}.
\end{equation}
On the other hand, from the fourth line of Algorithm~\ref{alg:dqds} we also have
\[ e_{n-1}^{(l+2)} = \frac{q_{n}^{(l+1)}
  e_{n-1}^{(l+1)}}{q_{n-1}^{(l+2)}}, \]
which implies
\[ q_{n}^{(l+1)}  =   \frac{e_{n-1}^{(l+2)} q_{n-1}^{(l+2)}}{
  e_{n-1}^{(l+1)}} \quad \mbox{and} \quad q_{n}^{(l)}  =
  \frac{e_{n-1}^{(l+1)} q_{n-1}^{(l+1)}}{ e_{n-1}^{(l)}} .\]
Plugging these equations into~(\ref{qn}), we obtain
\begin{equation}\label{Eqn:new}
{\displaystyle \frac{e_{n-1}^{(l+2)} q_{n-1}^{(l+2)}}{
  e_{n-1}^{(l+1)}} =   \frac{e_{n-1}^{(l+1)}
  q_{n-1}^{(l+1)}}{  \sqrt{e_{n-1}^{(l)}}}
\frac{\left(1+\|x^{(l)}\|\right)\sqrt{1+\xi^{(l)}}}{\left(1+\|x^{(l)}\|^2\right)
\sqrt{q_{n-1}^{(l)}}}-e_{n-1}^{(l+1)},}
\end{equation}
%By the remark after Lemma~\ref{thm:converg}, we have
%$e_{n-1}^{(l+1)} < e_{n-1}^{(l)}$ for large enough $l$, and
%hence
which can be rewritten as
\[{\displaystyle
\frac{e_{n-1}^{(l+2)} }{(e_{n-1}^{(l+1)})^{3/2}} = \eta^{(l)}, }  \]
with % (note that $\sqrt{e_{n-1}^{(l+1)}q_{n-1}^{(l+1)}} =\sqrt{e_{n-1}^{(l)}q_n^{(l)}}$ )
\[{\displaystyle
%\begin{split}
%\eta^{(l)}  & =  \frac{\sqrt{q_{n-1}^{(l+1)}}}{q_{n-1}^{(l+2)}} \left(
%\sqrt{\frac{e_{n-1}^{(l+1)}q_{n-1}^{(l+1)}}{e_{n-1}^{(l)}}}\frac{\left(1+\|x^{(l)}\|\right)\sqrt{1+\xi^{(l)}}}{\left(1+\|x^{(l)}\|^2\right)\sqrt{q_{n-1}^{(l)}}
%} - \sqrt{e_{n-1}^{(l+1)}}\right).\\
\eta^{(l)} = \frac{\sqrt{q_{n-1}^{(l+1)}}}{q_{n-1}^{(l+2)}} \left(
\frac{\sqrt{q_{n}^{(l)}}\left(1+\|x^{(l)}\|\right)\sqrt{1+\xi^{(l)}}}{\left(1+\|x^{(l)}\|^2\right)\sqrt{q_{n-1}^{(l)}}
} - \frac{\sqrt{e_{n-1}^{(l+1)}} }{\sqrt{q_{n-1}^{(l+1)}}}\right). }
%\end{split} }
\]

Since $\sqrt{1+\xi^{(l)}}\frac{1+\|x^{(l)}\|}{1+\|x^{(l)}\|^2}
\rightarrow 1$, $e_{n-1}^{(l+1)} \rightarrow 0$,
$q_{n}^{(l)}\rightarrow 0$, and $q_{n-1}^{(l+2)}$ converges to
$\sigma_{n-1}^2 - \sigma_n^2 > 0$, it follows that $\eta^{(l)}$
converges to $0$ as well.
\end{proof}

\bibliographystyle{plain}
\bibliography{thesis-refer}

\end{document}